\documentclass[12pt]{amsart}
\usepackage{amsmath,amssymb,amsfonts,amsopn,bm}
\usepackage{mathtools}
\usepackage{graphicx,color}
\usepackage{epstopdf}
\DeclareGraphicsExtensions{.png}
\usepackage{caption}
\usepackage{subcaption}
\usepackage{tikz}
\newtheorem{thm}{Theorem}
\newtheorem{lem}{Lemma}
\newtheorem{Def}{Definition}
\newtheorem{Prop}{Property}
\newtheorem{cor}{Corollary}
\newtheorem{rem}{Remark}
\def\R{\mathbb R}
\def\Z{{\mathbb Z}}
\def\N{{\mathbb N}}

\newcommand{\pt}[1]{\left( #1 \right) }
\newcommand{\ptg}[1]{\left\{ #1 \right\} }
\newcommand{\ptq}[1]{\left[ #1 \right] }
\newcommand{\ptqt}[1]{\left[ #1 \right) }
\newcommand{\abs}[1]{\left| #1 \right| }
\newcommand{\norm}[1]{\left\| #1 \right\| }
\newcommand\dist[2]{\mathrm{d}\pt{#1,#2}}

\def\supp{\operatorname{supp}}

\newcommand{\Lprod}[2]{ \langle #1, #2\rangle }

\def\phipkn{\varphi_{j,k,\lambda}}

\def\F{\operatorname{F}}

\def\Phip{\Phi_{j,k}}
\def\pii{2\pi \mathrm{i}\,}

\def\bap{\beta\pt{j}}

\def\varb{j,k,\lambda\in\Gamma}

\def\Dil{\operatorname{D}}%Dilatazione
\def\bap{2^j}
\def\bas{2^{\sigma}}

\def\varb{j,k,\lambda\in\Gamma}
\def\ud{\mathrm{d}}
\def\varpb{\varphi_{\bullet}}
\renewcommand{\phi}{\varphi}
\def\sinc{\operatorname{sinc}}

\begin{document}
\begin{abstract}
  We construct a family of 
  frames describing Sobolev norm and Sobolev seminorm of the space $H^s(\R^n)$.
  Our work is inspired by the Discrete Orthonormal Stockwell Transform
  introduced by R.G. Stockwell, which provides a time-frequency localized version of Fourier basis of $L^2([0,1])$.
  This approach is a hybrid between Gabor and Wavelet frames.
  We construct explicit and computable examples of these frames, discussing their properties.
\end{abstract}
\vspace{1em}
%%%%%%%%%%%%%%%%%%%%%%%%%%%%%%%%%%%%%%%%%%%%%%%%%%%%%%%%%%%%%%%%%%%%%%%%%%%%%%%%%%%%%%%%%%%%%%%%%%%%%%%%%%%%%%%%%%%%%%%%%%%%
%%%%%%%%%%%%%%%%%%%%%%%%%%%%%%%%%%%%%%%%%%%%%%%%%%%%%%%%%%%%%%%%%%%%%%%%%%%%%%%%%%%%%%%%%%%%%%%%%%%%%%%%%%%%%%%%%%%%%%%%%%%%
%title

\title{Stockwell-Like Frames for Sobolev Spaces}

\author{Ubertino Battisti, Michele Berra, Anita Tabacco}
\address{Dipartimento di Scienze Matematiche,
Politecnico di Torino, corso Duca degli Abruzzi 24, 10129 Torino,
Italy}
\email{ubertino.battisti@polito.it}
\address{Dipartimento di Scienze Matematiche,
Politecnico di Torino, corso Duca degli Abruzzi 24, 10129 Torino,
Italy}
\email{michele.berra@polito.it}
\address{Dipartimento di Scienze Matematiche,
Politecnico di Torino, corso Duca degli Abruzzi 24, 10129 Torino,
Italy}
\email{anita.tabacco@polito.it}

\maketitle
  \section{Introduction}
%%%%%%%%%%%%%%%%%%%%%%%%%%%%%%%%%%%%%%%%%%%%%%%%%%%%%%%%%%%%%%%%%%%%%%%%%%%%%%%%%
%%%%%%INTRODUZIONE%%%%%%%%%%%%%%%%%%%%%%%%%%%%%%%%%%%%%%%%%%%%%%%%%%%%%%%%%%%%%%%
%%%%%%%%%%%%%%%%%%%%%%%%%%%%%%%%%%%%%%%%%%%%%%%%%%%%%%%%%%%%%%%%%%%%%%%%%%%%%%%%%
%A classical problem in time frequency analysis is the characterization of
%smoothness  spaces via suitable frames. The leading idea is that
%the smoothness or regularity must be characterized via decay or sparsity properties of
%associated discrete expansions. 
%{\color{red} MIK: 
\noindent One of the most challenging problems in modern mathematics is to represent key properties of signals using discretization techniques connected with numerical methods. Among these properties,  regularity (or smoothness) plays a very important role both in classical and applied mathematics. Therefore, it is interesting to characterize regularity function spaces,  via suitable frames. The leading idea is that
the smoothness or regularity must be characterized via decay or sparsity properties of
associated discrete expansions. \\
For example, it is well known that $H^s([0,1])$ can be characterized by the decay properties 
of the Fourier coefficients. Similarly, Sobolev spaces $H^s(\R)$, and more generally inhomogeneous Besov spaces,
can be described using suitable wavelets expansions, see for example \cite{ME92}. 
See \cite{BN07,CD02,FE89,FG85,FE85,Fo07,GR92,SCD02} for other discrete time-frequency representations of functional spaces.
In \cite{BB15},  the Stockwell frames are introduced associated to different partitions of
the real line, which are connected to the notion of $\alpha$-partitioning, see e.g. \cite{Fo07}. 
In this paper we  focus on the
Stockwell frames associated to a dyadic partition of the frequency domain. In order to give a clean introduction to our frame, we use 
a particular partition in dimension $d=1$, while later
we generalize to a wider set of admissible ones in arbitrary dimension. \\
Consider the system of functions
 \begin{equation}
 \begin{split}
    \phipkn(t)&=T_{\frac{\lambda }{\bap}}  \pt{  \frac{1}{ 2^{j/2} } \sum_{\eta \in Z_{j,k}} 
    e^{\pii \eta t } \varphi(t)},\quad  j\in \N, k=\pm, \lambda \in \nu \Z\\
   \label{eq:stframe} 
    \phi_{\bullet,\lambda}&=   T_{\lambda} \varphi(t), \quad \lambda \in \nu \Z
 \end{split}
 \end{equation}
 where $Z_{j,+}= 2^{j}, \ldots, 2^{j+1}-1$ and $Z_{j,-}=-2^{j+1}+1, \ldots, -2^{j}$.\medskip\\
 The first aim of the paper is to study in detail the system of functions \eqref{eq:stframe} 
 and to establish conditions so that this system is a frame characterizing Sobolev norms. We also generalize~\eqref{eq:stframe} to arbitrary dimension.
In Theorem \ref{thm:fram}, we prove that if $\varphi$ satisfies the conditions of
 Definition \ref{def-sadm} for a certain $s\geq 0$ then the system of functions \eqref{eq:stframe}, for $\nu$ small enough, 
 is a frame of $L^2(\R)$ which characterize the $H^s$-Sobolev norm, 
 that is
 \begin{align}
   \label{eq:framsob}
   A\norm{f}^2_s \leq\sum_{\lambda\in\nu\Z}|\Lprod {f}{\phi_{\bullet,\lambda}}|^2  + 
   \sum_{j\in \N, k=\pm, \lambda \in \nu \Z}2^{2js}\abs{\Lprod{f}{\phipkn}}^2  \leq B\norm{f}^2_s.
 \end{align}
 The requirements of Definition \ref{def-sadm}  involve the decay and non vanishing properties of suitable linear combinations
 of translated of the Fourier transform 
  $\widehat{\phi}$. 
 Given $s\geq 0 $, it is not difficult to find functions which satisfy Definition \ref{def-sadm}.
 For example the Gaussian works for all $s$. 
 In Theorem \ref{thm:condsuff}, we provide natural sufficient conditions on
 ${\phi}$ so that it defines a frame describing the $H^s$-norm. 
 In Section \ref{sec:ex},
 we give some explicit examples and, in a particular case, we show that the system of functions \eqref{eq:stframe}
 is indeed an orthonormal basis of $L^2(\R)$, characterizing all Sobolev spaces.

In some cases, approximating the norm is not enough to capture the regularity properties of a signal and a finer analysis is needed. In this perspective, the analysis of the Sobolev's seminorms is a natural task to address with frames.  
As we mentioned before, wavelet frames, under suitable conditions, can describe the $H^s$-norm. 
 Actually, in several cases, it is possible to prove that wavelets characterize the $H^s$-seminorm as well, see \cite{ME92} and references therein.
 It is therefore an interesting question to determine conditions under which the Stockwell frame 
 describes not only the $H^s$-norm, but 
 also the $H^s$-seminorm, which is a key tool
  in order to describe Besov spaces and other interpolation between Sobolev spaces.
 In Theorem \ref{thm:mainsemi}, we prove a seminorm characterization using a frame arising from \eqref{eq:stframe}. \\ We set $\phipkn(t)$ as in \eqref{eq:stframe} and we add (for negative values of $j$)
 \begin{align}
%   \nonumber
%   \phipkn(t)&=T_{\frac{\lambda }{\bap}}  \pt{  \frac{1}{ 2^{j/2} } \sum_{\eta \in Z_{j,k}} 
%    e^{\pii \eta t } \varphi(t)},\quad  j\in \N, k=\pm, \lambda \in \nu \Z\\
   \label{eq:stframsemi}
   \phi_{-j,k,\lambda}(t)&= T_{\lambda \bap} \Dil_{2^{2j}} \phi_{j,k,0}(t), \quad  j\in \N\setminus \ptg{0}, k=\pm, \lambda \in \nu \Z
 \end{align}
where $\Dil_a(f)(\cdot)= \frac{1}{\sqrt{a}} f\pt{ \frac{\cdot}{a} }$ is the usual $L^2$-dilation operator.
We prove that if $\phi$ satisfies the conditions of Definition \ref{def-sadmsemi} then 
\[
  A\abs{f}^2_s \leq 
   \sum_{j\in \Z, k=\pm, \lambda \in \nu \Z}2^{2js}\abs{\Lprod{f}{\phipkn}}^2  \leq B\abs{f}^2_s.
\]
As expected, the conditions of Definition \ref{def-sadmsemi} are stronger than those of the $H^s$-norm
characterization. Nevertheless, it is possible to provide explicit examples of frames describing the $H^s$-seminorm, see Section \ref{sec:ex}.
% The condition of $s$-admissibility for the $H^s$-seminorm
% is actually linked to the existence of so called \emph{interpolatory functions}, precisely $f$ so that 
% $f(0)\neq 0$ and $f(m)=0$ for all $m\in \Z\setminus \ptg{0}$. This kind of functions is a key tool in the construction of
% Daubechies wavelets, \cite{DA88}. In \cite{ME92}, these functions are called \emph{Lagrangian Spline}.

 The frames \eqref{eq:stframe} and \eqref{eq:stframsemi} can be extended in a natural way to 
 the $n$-dimensional case.
The definition  is truly multidimensional, in the sense that the frame elements are, in general, not tensor product 
of $1$-dimensional frames. Indeed, in Section~\ref{sec:normHs} we show how to construct admissible
partitions in $\R^d$. The frame results we stated for the Sobolev norms and semi-norms in the one dimensional case 
are actually valid in 
arbitrary dimension.

It is worth to explain why we choose \eqref{eq:stframe} as candidate to be a $L^2$-frame,
and why the dyadic partition should be the correct one to describe Sobolev Spaces.

In \cite{BR16}, it is proven that the so called DOST-functions, 
\begin{equation}
\begin{split}
  P_{j,k,\tau}(t)&= T_{\frac{\tau}{2^{j}}}\frac{1}{2^{j/2}} \sum_{j\in Z_{j,k}}  e^{\pii \eta t}, \quad j\in \N, \tau=0, \ldots, 2^{j}-1, k=\pm, \quad 
  P_{\bullet}(t)=1
  \label{eq:dost}
\end{split}	
\end{equation}

form an orthonormal basis of $L^2([0,1])$ which is time-frequency localized in terms of the Donoho-Stark uncertainty principle \cite{DS89}. 
The DOST functions \eqref{eq:dost} were first introduced in \cite{ST07} as a discretization of the $S$-transform, 
defined in \cite{ST96}. In \cite{WA09}, the DOST functions were studied in order to obtain a FFT-fast algorithm able to compute the related coefficients, that is the scalar product $\Lprod{f}{P_{j,k,\tau}}$.

The naive idea of the definition of the Stockwell frames comes from the interpretation of Gabor frames
\[
  T_{\lambda} \pt{e^{\pii m t} \phi(t)}, \quad \lambda  \in \nu \Z, m\in \Z
\]
as the uniform translation of $\ptg{e^{\pii mt}}_{m\in \Z}$, the usual Fourier basis of $L^2([0,1])$, localized via the window function $\phi$.
Using the DOST basis instead of the  Fourier one,  and a natural non-stationary translation we are led to
\begin{align}
  \nonumber
     T_{\frac{\lambda }{\bap}} \pt{ P_{j,k,0}(t) \phi(t)}, \quad  
\textrm{ and }\quad       T_{\lambda} P_{\bullet}\varphi(t)=  T_{\lambda} \varphi(t), \quad j\in \N, k=\pm, \lambda \in \nu \Z
\end{align}
which is exactly the system of functions defined in \eqref{eq:stframe}.
Notice that the non-stationary translation is related to the frequency parameter $j$. Roughly, we refine the space translation as the frequency increases. \\ 
Since Sobolev spaces are naturally associated to a dyadic partition of the frequency domain, we used the same approach for our frequency tiling.  \medskip\\
 The time frequency method which is behind the construction of Stockwell frames is the $S$-transform.
 The $S$-transform has been  studied in several papers
 both from an applied point of view (see \cite{DR09, ZH14}) and a theoretical one (see \cite{WO10,RW13}); it can be seen as an hybrid between the Continuous Wavelet Transform and the 
 Short Time wavelet transform, see \cite{GLM06,VE08}. 
 This \emph{hybrid} behavior can be described also in terms of representation theory,
 indeed, the Continuous Wavelet Transform hides a representation of the affine group and the Short Time Fourier Transform is related to
 the Heisemberg group, while the $S$-transform is essentially represented by the affine Weyl Heisemberg group, see
 \cite{RI14,RW13} and also \cite{DA08,kalisa1993n} for other applications of the affine Weyl Heisember group.
 
This \emph{hybrid} nature is the reason why the Stockwell frames shares several properties of Gabor frames
  and of wavelets frames. Indeed, our proof are much closer to the Gabor frames techniques, while the time frequency partitioning and the space characterization properties are closer to the wavelets approach.
  Compared with wavelet theory, 
  we do not use dilations to increase localization in time,
  while sums of translates in the frequency domain. In same cases the two approaches are very similar,
  for instance if we compare the Stockwell-like frame with the $\sinc$ as window function and the Shannon wavelet.
  Other cases may look very different; moreover, with this approach we can use a wider class of window functions,
  for example the Gaussian, as we shall see in Section \ref{sec:ex}. 
%  \label{sec:int}

   \section*{Notations}
 We choose the following normalization for the Fourier Transform
  \[
   \pt{\F u}(\omega)=\int e^{-2 \pi i x\cdot \omega } u(x)dx.
  \]
  We denote with $\norm{\cdot}_s$ and $\abs{\cdot}_{s}$ respectively the $H^s$-norm and seminorm.\\
  We write 
  \[
    f\lesssim g,   
  \]
  is there exists a constant $C\in \R$ such that
  \[
   f\leq C\, g. 
  \]
  We write $f\sim g$ if $f\lesssim g$ and $g \lesssim f$.

  \section{Frames of $H^s(\R^d)$}
   \label{sec:normHs}
   
   \noindent 
   First, we define a dyadic frequency  tiling of $\R^d$; 
   then we construct  frames for the $H^s$-norm associated to admissible partitions.
   \subsection{Admissible partitions of $\R^d$}
   We state here the general definition in arbitrary dimension then we give some
   examples of admissible partitions. Similar concepts
   are used for example in \cite{FG85,Fo07}, and we refer also to the bibliography presented there.
      \begin{Def}
        \label{def:adm}
        The family $\ptg{\ptg{I_{j,k}}_{j\in \N, k\in K}, I_{\bullet}}$, where $K$ is a finite index set,
        is called admissible if
        \begin{itemize}
          \item [i)] $I_{j,k}$ and $I_{\bullet}$ are non empty connected subsets of $\R^d$; 
          \item [ii)] $\bigcup_{j\in \N, k\in K } I_{j,k} \cup I_{\bullet}= \R^d$;
          \item [iii)] $I_\bullet$ is a neighborhood of the origin and for each $j\in \N, k\in K$ $I_{j,k}\cap \Z^d\neq \emptyset$;
          \item [iv)] has the finite intersection property, that is there is $N\in \N$ such that 
          for each $\pt{\bar{j}, \bar{k}}$ there exist at most $N$ indexes $(j,k)$ such 
          that $I_{\bar{j},\bar{k}}\cap I_{j,k}\neq \emptyset$;
          \item [v)] There exists $c_{\min},c_{\max}>0$ such that for all $\omega\in   I_{j,k}$
          \[
                       c_{\min}<\frac{|\omega|}{2^j}<c_{\max}, \quad \mbox{uniformly in } j,k.
                     \]
                    where $|I_{j,k}|$ is the Lebesgue measure of $I_{j,k}$. 
        \end{itemize}
   \end{Def}
   \begin{rem}
   Essentially, an admissible partition covers $\R^d$ with no holes and
   with finitely many intersections and each set contains at least some integer. 
   Moreover, the elements of the set $I_{j,k}$ are comparable with the scaling $2^j$. 
   The index $k$ represents the direction in frequency and, by taking K to be finite, we choose to restrict our frame to a finite number of directions.	\medskip\\
   \end{rem}
   \paragraph{\textbf{Dimension $d=1$}}
   We set
   \begin{equation}\label{eq:setspl}
     \begin{split}	
     I_{j,+} &= \ptg{\omega\in\R, \omega\in \ptqt{\bap-\frac{1}{2}, 2^{j+1}-\frac{1}{2} }},
     \\ I_{j,-} &= \ptg{\omega\in\R, \omega
     \in \left(-2^{j+1}+\frac{1}{2}, -2^j +\frac{1}{2}\right]  },\\
     I_{\bullet} &=\pt{-\frac{1}{2},\frac{1}{2} }.	
   \end{split}
   \end{equation}
   Clearly this partition is admissible in the sense of Definition~\ref{def:adm}; it is worth to notice that with this particular choice we show in Section~\ref{sec:ex} that it is possible to define a Stockwell-like frame that 
   is also an orthonormal basis of $L^2(\R)$.
   Finally, we notice that in dimension $d=1$ we have just two possible directions, represented by $k = \pm$.\medskip\\
   \paragraph{\textbf{Dimension $d=2$}}
   In order to extend the case $d=1$, 
   let us work in polar coordinates. We consider balls of radius proportional to $2^j$ and we define eight possible directions, as shown in Figure \ref{F:part}.
   The precise definition is simpler if we use polar coordinates 
   \[\omega(\rho,\theta) = \pt{\rho \cos(\theta), \rho \sin(\theta)}.\]
   Set
   \begin{equation}\label{eq:setspl_d}
     \begin{split}	
     I_{j,k} &= \ptg{\omega\in \R^2 : 2^j -\frac{1}{2}\leq \rho <2^{j+1} -\frac{1}{2},\: \frac{k\pi}{4}\leq\theta<\frac{(k+1)\pi}{4} }, \quad k=0,\ldots,7
   \\  I_{\bullet} &=\ptg{\omega\in \R^2 : \rho <\frac{1}{2}},
   \end{split}
   \end{equation}
   This partition extends naturally to arbitrary dimension. 
   % \begin{rem}
   % We notice that the choice of eight directions is not accidental. Indeed, in order to fulfill $iii)$ of Definition~\ref{def:adm}, we need to have non empty intersections with $\Z^2$ for each set. If we take $j=0$, then the number of integers of modulo in $(-\frac{1}{2},\frac{1}{2})$ is exactly eight. Hence, taking more directions violates this requirement in dimension $d=2$. 
   % \end{rem}
   \begin{figure}
   \centering
   \begin{tikzpicture}[>=latex]
    \draw [dashed,domain=90:405] plot ({1.5*cos(\x)}, {1.5*sin(\x)});
     \draw [thick,domain=0:360] plot ({.5*cos(\x)}, {.5*sin(\x)});
     \draw [dashed,domain=90:405] plot ({3.5*cos(\x)}, {3.5*sin(\x)});
       \draw [thick,domain=45:90] plot ({1.5*cos(\x)}, {1.5*sin(\x)});
     \draw [thick,domain=45:90] plot ({3.5*cos(\x)}, {3.5*sin(\x)});
   \node at (.95,2.31) [left] {$I_{j,k}$};
   \node at (.95,2.31){.};
   \node at (0.65,0) [below] {\scriptsize $\frac{1}{2}$};
   \node at (1.65,0) [below] {\scriptsize $\frac{3}{2}$};
   \node at (3.65,0) [below] {\scriptsize $\frac{7}{2}$};
   \draw[dashed]  (.35,.35) -- (2.48,2.48);
   \draw[thick]  (1.05,1.05) -- (2.48,2.48);
   \draw[dashed]  (-.35,-.35) -- (-2.48,-2.48);
   \draw[dashed]  (.35,-.35) -- (2.48,-2.48);
   \draw[dashed]  (-.35,.35) -- (-2.48,2.48);
   \draw[thick]   (0,1.5) --(0,3.5);
   \draw[dashed]   (0.5,0) --(3.5,0);
   \draw[dashed]   (-.5,0) --(-3.5,0);
   \draw[dashed]   (0,-.5) --(0,-3.5);
   \draw[dashed]   (0,.5) --(0,1.5);
   \end{tikzpicture} 
   \caption{An example of admissible partition in $d=2$.}
   \label{F:part}
   \end{figure}
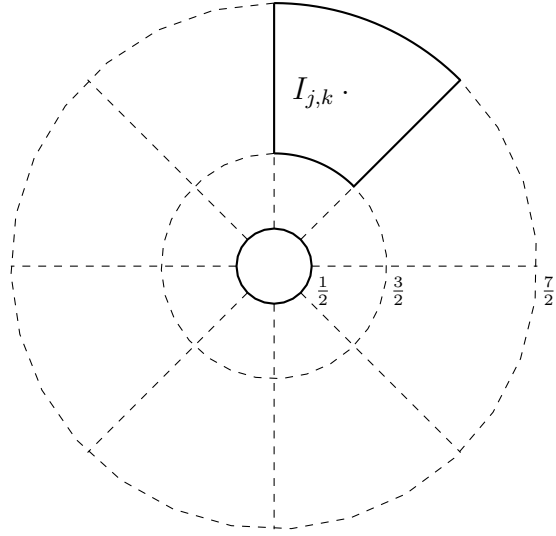

   \begin{figure}
     \centering
     \begin{minipage}{0.4\textwidth}
       \includegraphics[width=1\textwidth]{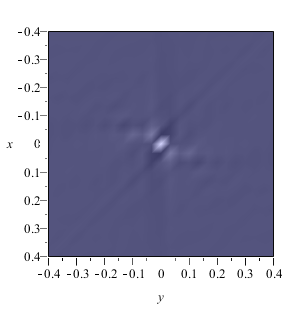}
     \end{minipage}
     %\hfill
     \begin{minipage}{0.4\textwidth}
       \includegraphics[width=0.87\textwidth]{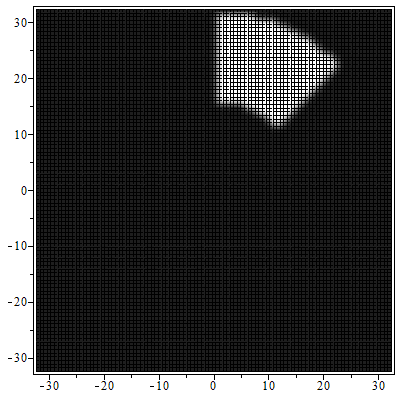}
     \end{minipage}
 
    \caption{A frame element obtained using the partition of Figure \ref{F:part}, the scale parameter is $j=4$
    the angular parameter is the same of Figure \ref{F:part}, that is $k=1$. As localizing window we consider
    $\phi(x,y)=\sinc^3(x)\cdot \sinc^3(y)$. On the left is plotted the real part of the frame element and on the right 
    the absolute value of its Fourier transform.
    }
    \label{Fig:2dframe}
   \end{figure}

   \subsection{Stockwell-like frames definition}
   Consider the functions $\varphi, \varpb \in L^2(\R^d)$, and the index set 
   \[
   \Gamma = \ptg{(j, k , \lambda) \mid j\in \N, k\in K, \lambda \in \nu\Z^d},
   \]
   with $\abs{K}<+\infty$ and an admissible partition $\ptg{\ptg{I_{j,k}}_{j \in \N, k \in K}, I_\bullet}$.
   %We set
   %\begin{equation}\label{eq:setspl}
   %  \begin{split}	
   %  I_{j,+} &= \ptg{\omega\in\R, \omega\in \pt{\bap-\frac{1}{2}, 2^{j+1}-\frac{1}{2} }}, \quad I_{j,-} = \ptg{\omega\in\R, \omega
   %  \in \left(-2^{j+1}+\frac{1}{2}, -2^j +\frac{1}{2}\right)  },\\
   %  I_{\bullet} &=\pt{-\frac{1}{2},\frac{1}{2} }.	
   %\end{split}
   %\end{equation}
   Then, the Stockwell frame as defined in \cite{BB15}, is 
   \begin{align*}
      \phipkn(t)=&T_{\frac{\lambda}{\bap}}  \pt{  \frac{1}{ 2^{jd/2} } \sum_{\eta\in Z_{j,k}} 
     e^{\pii \eta \cdot t } \varphi(t)}, & j,k,\lambda \in \Gamma,
   \\  
     \varphi_{\bullet,\lambda}(t) =& \varpb(t - \lambda),
   \end{align*}
    where $Z_{j,k}  = I_{j,k}\cap\Z^d$.
   Taking the Fourier transform yields
    \begin{align*}
      \widehat{{\varphi}_{\bullet,\lambda}} (\omega) =& e^{-\pii \omega\cdot \lambda} \widehat{\varpb}( \omega ) & \\
      \widehat{\phipkn}(\omega)=
      & e^{-\pii \omega\cdot  \frac{\lambda}{\bap}} \pt{  \frac{1}{ 2^{jd/2}} \sum_{\eta\in Z_{j,k}} 
      \widehat{\varphi}(\omega-\eta)}.\\
    \end{align*}
   % Let us define 
   % \begin{align*}
   %   \Phi_{\bullet}(\omega)&=\widehat{\varpb}(\omega), \quad &\\
   %   \Phi_{j,k}(\eta)&=\sum_{\eta\in Z_{j,k}} 
   %   \widehat{\varphi}(\xi-\eta).
   % \end{align*}
   We remark that the exact $L^2$-norm normalization constant for the frame elements
   is $\abs{Z_{j,k}}^{-1/2}$ rather than $2^{-jd/2}$. Since $\abs{Z_{j,k}}\sim 2^{jd}$
   we chose the latter for the sake of clarity.

   \medskip

   If $\varphi$ is a localization window, 
   the sum over the index set $Z_{j,k}$ determines, in the space domain, a higher localizations near the point
   $\frac{\lambda}{2^j}$ as  
   $j$ increases while, in the frequency domain, it implies that the frame element $\phi_{j,k,\lambda}$
   has a Fourier transform $\widehat{\phipkn}$ with the \emph{ major part} of the support in the set $I_{j,+}$.
   %see Figure \ref{Fig:TreGauss}.
   In Figure \ref{Fig:FrWin} we plot some frame elements with different windows $\phi$, 
   in the one dimensional case. 
   In Figure \ref{Fig:2dframe}, we plot a frame element and its Fourier transform in the $2$-d case using the partition
   of Figure \ref{F:part}.
   %As pointed out in the introduction, our frame is composed by an (adaptive) translation of two factors: a window $\varphi$ that ideally restricts our frame to $\ptq{-\frac{1}{2},\frac{1}{2}}$ and a sum of harmonics in the frequency band. This sum of Fourier exponents is responsible for oscillations and time concentration; see Figure \ref{Fig:FrWin} for a plot of these frame elements. 

   \begin{figure}
       \centering
       \begin{subfigure}[b]{0.4\textwidth}
   \includegraphics[width=.7\textwidth]{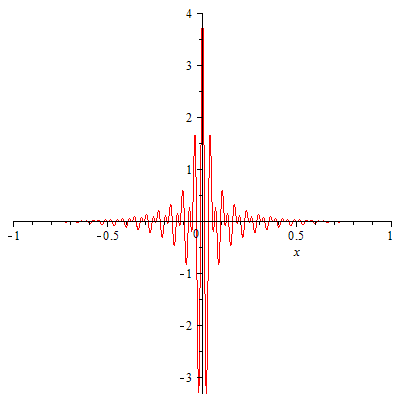}
           \caption{Time, $\phi_{4,+,0}(t) $ \\with $\phi(t)=\sinc^3(t)$}
       \end{subfigure}
       \begin{subfigure}[b]{0.4\textwidth}
           \includegraphics[width=.7\textwidth]{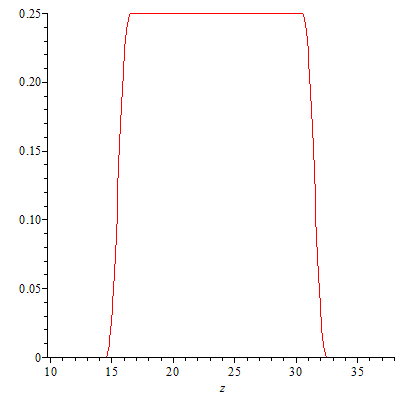}
           \caption{Frequency, $\widehat{\phi_{4,+,0}}(\omega)$ \\
           with $\phi(t) = \sinc^3(t)$}
       \end{subfigure}\\
    
           \begin{subfigure}[b]{0.4\textwidth}
   \includegraphics[width=.7\textwidth]{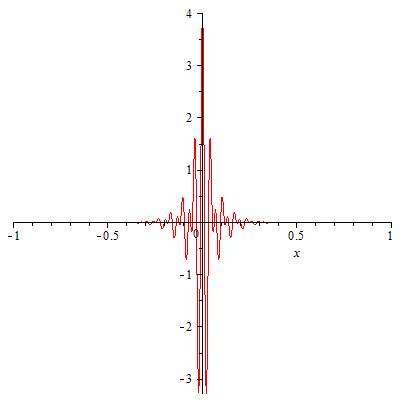}
           \caption{Time, $\phi_{4,+,0}(t)$\\ with $\phi(t)= \sinc^{15}(t)$}
       \end{subfigure}
       \begin{subfigure}[b]{0.4\textwidth}
           \includegraphics[width=.7\textwidth]{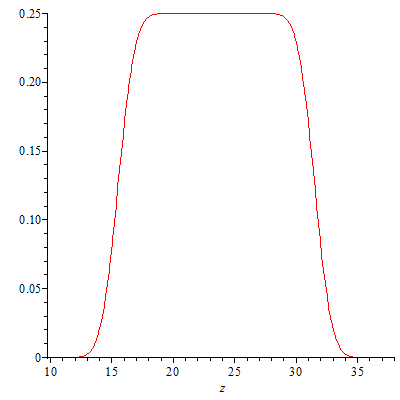}
           \caption{Frequency, $\widehat{\phi_{4,+,0}}(\omega)$\\ with $\phi(t) = \sinc^{15}(t)$}
       \end{subfigure}
   \caption{Frame element $\phipkn(t)$ in both time and frequency in dimension $d=1$. 
   Here, the reference frequency is $2^j, j=4$. We observe that the frame window in frequency \emph{works} as
   a characteristic
   function over $\ptq{2^j-\frac{1}{2}, 2^{j+1}-\frac{1}{2}}$. The frame functions are in general complex, here we are plotting the real part only.
   }\label{Fig:FrWin}
   \end{figure}

   First, we introduce conditions under which the system of functions
    $\ptg{\phipkn}_{j,k,\lambda\in\Gamma}$ is a frame representing the $H^s$-norm. Later, we 
    analyze $H^s$-seminorm's case as well. 
 
    \subsection{Main assumption and result}   
   Before going into specific details, let us give a brief explanation of which conditions we need
   for the $H^s$-norm characterization.
   We require that the frame elements
   cover in the frequency domain the whole of $\R^d$ and that they do not overlap too much. 
   Finally, we require  decay properties at infinity in the frequency domain, in order to obtain
   Sobolev regularity. 
   %Figure \ref{Fig:TreGauss} gives a hint of these two ideas. 
   Let us formalize. 
   %%%%%%%%%%%%%%%%
   \begin{Def}[$s$-admissible for Sobolev norms]
   	\label{def-sadm}
   Let $\ptg{\ptg{I_{j,k}}_{j \in \N, k \in K}, I_\bullet}$ 
   be an admissible covering and $s\geq0$, 
   we say that a pair of functions $\varpb,\varphi$ is $s$-admissible with respect to the covering if, given
   \begin{align}
   \Phi_{\bullet}(\omega) = \widehat{\varpb}(\omega),\qquad \Phip(\omega)=\sum_{\eta\in Z_{j,k} }
                         		\hat{\varphi}(\omega- \eta)\label{eq:G}
   \end{align}
   then
       \begin{itemize}
         \item[i)]There exists $\alpha>\frac{d}{2}$ such that 
   	\begin{align}
   |\Phi_{\bullet}(\omega)| \lesssim \frac{1}{(1+|\omega|)^{\alpha}},\quad\textrm{ and }\quad
           \label{eq:con2}
           |\Phip(\omega)|\lesssim \left\{
   			    \begin{array}{ll}
   			      \frac{2^{jd/2}}{\pt{1+d(\omega, I_{j,k})}^{\alpha+s}},& \quad \omega\neq I_{j,k}\\
   			      1, &\quad \omega\in \R^d
   			    \end{array}
   			   \right. . 
          \end{align}
   \item[ii)] There exists a constant $a>0$ such that,  if $\omega \in I_{\bullet}$ then 
   \begin{equation}
     \label{eq:cond2bullet}
     \abs{\Phi_{\bullet}(\omega)} \geq a,
   \end{equation}
   otherwise, there exists $\bar{j},\bar{k}$ such that $\omega \in I_{\bar{j},\bar{k}}$  and
     \begin{align}
   	   \label{eq:cond2j}
   	   |\Phi_{\bar{j},\bar{k}}(\omega)|\geq a.
     \end{align}
     The constant $a$ does not depend on $\bar{j},\bar{k}$.
       \end{itemize}	
   \end{Def} 
   \begin{rem}
   The decay hypothesis on $\varpb$ implies
   \[
   \sum_{\lambda \in \nu\Z} |\Lprod {T_{\lambda} \varpb}{f}|^2\lesssim \norm{f}^2_{L^2(\R^d)},	
   \]
   see e.g. \cite{CH03}[Thm 7.2.3 p.143].
   \end{rem}

   We state the main result for our Stockwell-like system. The $L^2(\R^d)$ case, hereby represented by $s=0$, has been proven already in \cite{BB15} under a different and stronger set of 
   hypotheses.
   \begin{thm}\label{thm:fram}
   Consider a $s$-admissible set of functions $\varpb,\varphi$ - cf. Definition~\ref{def-sadm} - with $s\geq 0$. Then there exists $\nu_0>0$ such that for each $\nu\in (0,\nu_0)$ the system
   \[
   F\pt{\varpb, \varphi,\Gamma}=\ptg{T_{\lambda}\varpb,\: \lambda \in \nu\Z^d} \cup \ptg{\phipkn,\: j,k,\lambda \in \Gamma}
   \] 
   is a frame representing the $H^s(\R^d)$ norm.
   Precisely, there exists $A,B>0$ such that for each $f\in H^s(\R^d)$
   \[
   A\norm{f}^2_s \leq\sum_{\lambda\in\nu\Z}|\Lprod {f}{T_{\lambda} \varpb}|^2  + \sum_{\varb}2^{2js}\abs{\Lprod{f}{\phipkn}}^2  \leq B\norm{f}^2_s.
   \]
   \end{thm}
   We prove the result in the next sections providing first the upper bound and then we deal with the lower one. 
   We develop the proof in dimension $d=1$ for the sake of simplicity. The proof for arbitrary dimension is essentially the same. 
   % \begin{figure}
   % 	  \includegraphics[width=.3\textwidth]{3gaussianeFOurier}
   %                 \caption{ Real part of the Fourier transforms 
   %                 $\widehat{\phi_{3,+,0}}(\omega)$,
   %                 $\widehat{\phi_{4,+,0}}(\omega)$,
   %                 $\widehat{\phi_{5,+,0}}(\omega)$
   %                 with a Gaussian window $\phi$.}	
   %                   \label{Fig:TreGauss}
   % \end{figure}

   %
   %\begin{rem}
   %We used $2^j$ as the (absolute value of the) reference frequency of $I_{j,k}$ even if it does not represent the center of the frequency band $I_{j,k}$. In principle, the correct Sobolev weight is given by $(1+|\omega_{j,k}|)$ where 
   %$\omega_{j,k} = k(2^{j}+2^{j/2})$ is the center of the band $I_{j,k}$.  Clearly,
   %\[1+|\omega_{j,k}| \asymp 2^j,\]
   %hence, for the sake of clarity, we shall use $2^j$ as our frequency weight.
   %\end{rem}

   \subsection{Upper bound}
   We  introduce a fundamental lemma.
   \begin{lem}\label{L_uplow}
       Let $s\geq 0$ and $\varphi $ be a function such that
         \begin{align}
           \label{eq_fcnd}
           |\Phi_{j,k}(\omega)|\lesssim \left\{
   			    \begin{array}{ll}
   			      \frac{2^{j/2}}{\pt{1+d(\omega, I_{j,k})}^{\alpha+s}},& \quad \omega\notin I_{j,k}\\
   			      1, &\quad \omega \in \R
   			    \end{array}
   			   \right. ,
          \end{align} 
          for some $\alpha>\frac{1}{2}$.
   Then
     \begin{align}
       \label{eq:bd_ups1}    
        \frac{2^{js}}{(1+|\omega|)^{s}} |\Phip(\omega)| &\lesssim \frac{2^{j/2}}{\pt{1+d(\omega, I_{j,k})}^{\alpha}}  &\quad \mbox{a.e. }\omega \notin I_{j,k}, 
        \\
        \label{eq:bd_ups2}
        \frac{2^{js}}{(1+|\omega|)^{s}} |\Phip(\omega)| &\lesssim 1  &\quad \mbox{a.e. }\omega \in I_{j,k}.      
     \end{align}

   \end{lem}
   \begin{proof}
   	Inequality~\eqref{eq:bd_ups2} is trivially verified since $\omega \sim 2^j$ if $\omega \in I_{j,k}$.

    In order to prove \eqref{eq:bd_ups1}, notice that
    by hypothesis \eqref{eq_fcnd}
    \begin{align*}
         \frac{2^{js}}{(1+|\omega|)^s} \abs{\Phip(\omega)}
         %\lesssim 
         %\frac{2^{js}}{(1+|\omega|)^s}  \frac{2^{j/2} } {\pt{1+d(\omega, I_{j,k})}^{\alpha+s}}\\      
         &\lesssim \frac{2^{js}}{(1+|\omega|)^s \pt{1+d(\omega,I_{j,k})}^s} \frac{2^{j/2}}{\pt{1+d(\omega,I_{j,k})}^{\alpha}}.\\
       \end{align*}
   Now, there exists $\overline{\omega}\in I_{j,k}$ such that 
   $d(\omega,I_{j,k}) = |\omega -  \overline{\omega}|$, moreover 
   $|\overline{\omega}|\sim 2^j$ because the covering is admissible.
   Hence, by triangular inequality 
    \begin{align*}
         \frac{2^{js}}{(1+|\omega|)^s} \abs{\Phip(\omega)}&
         \lesssim \frac{2^{js}}{(1+|\omega|)^s \pt{1+|\overline{\omega}-\omega|}^s} \frac{2^{j/2}}{\pt{1+d(\omega,I_{j,k})}^{\alpha}}\\
   	&\lesssim \frac{2^{js}}{ \pt{1+|\overline{\omega}|}^s} 
   	\frac{2^{j/2}}{\pt{1+d(\omega,I_{j,k})}^{\alpha}}\\
   	&\lesssim \frac{2^{j/2}}{\pt{1+d(\omega,I_{j,k})}^{\alpha}}.
       \end{align*} 
   \end{proof}
   We have an immediate corollary.
   \begin{cor}
   \label{cor:sommanorm}
   Under the hypothesis \eqref{eq_fcnd} of Lemma~\ref{L_uplow} above, there exists $ b_s \in \R$ such that
     \begin{align}
       \label{eq:ups}
        \sum_{j,k\in\Gamma}\frac{2^{2js}}{(1+|\omega|)^{2s}} |\Phip(\omega)|^2\leq b_s, \quad \mbox{a.e. }\omega \in \R. 
     \end{align}	
   \end{cor}
     \begin{proof}
       For any $\omega \in \R$, there exists $\bar{j}, \bar{k}$ such that $\omega\in I_{\bar{j}, \bar{k}}$
       and for all $\pt{j,k}$ except $\pt{\bar{j}-1,\bar{k}}$, $\pt{\bar{j}+1,\bar{k}}$ we have
     \begin{align*}
     	\dist{\omega}{I_{j,k}} > 2^{j-1},
     \end{align*}  
     this number is independent from the particular point $\omega$. Hence, the claim follows from Lemma~\ref{L_uplow}.
     \end{proof}
   \begin{rem}\label{rem_decay_finite}
   	We notice that the decay at infinity in~\eqref{eq_fcnd}
    can be relaxed. Precisely, we may ask
    that $\Phip$ is uniformly
    bounded and that \eqref{eq_fcnd} holds for $j\geq j_0$, for some $j_0>0$. Indeed, when we estimate the sum in~\eqref{eq:ups}, we can ignore an arbitrary (but finite) number of terms.
   \end{rem}
   In order to show the boundedness of the coefficients with respect to the Sobolev norm, we introduce the following family of sets
   \begin{equation}
   \label{eq:epnu}
   \begin{split}
   E_{j,k}&= \ptg{\omega \in \R^d : \dist{\omega}{I_{j,k}}\leq 2^{j-1}},
   \\
   E_{\bullet} &= I_{\bullet}. 
   \end{split}
   \end{equation}

    \begin{lem}
     \label{lem:1Bessel}
      Let $s\geq 0$ and $\varphi$  such that $\Phip(\omega)$ satisfies hypothesis \eqref{eq_fcnd}.
   Define
      \begin{align}
         \label{eq:widetilde}
         \widetilde{\phi_{j,k,\lambda}}(t)=\frac{1}{2^{j/2}} T_{\frac{\lambda}{\bap}} 
         \F^{-1}_{\omega\mapsto t}\pt{ \frac{2^{js}}{(1+|\omega|)^s} \Phip(\omega)}(t)
      \end{align}
      then for $\nu\in (0,1]$ and all $j,k$ the system of functions 
      \[
        \ptg{ \widetilde{\phi_{j,k,\lambda}}(t)}_{\lambda\in \nu \Z}
      \]
      is a Bessel sequence uniformly in $j,k$, that is
      \begin{align}
        \label{eq:bessel}
        \sum_{\lambda \in \nu\Z} \abs{\Lprod{\widetilde{\phi_{j,k,\lambda}}} {f}}^2\leq C_{\nu} \norm{f}^2_{L^2(\R)}
      \end{align}
      with $C_{\nu}$ independent on $j,k$.
    \end{lem}
   %%%%%%%%%%%%%%%%%%%%%%%%%%%%%%%%%%%%%%%%%%%%%%%%%%%%%%%%%%%%%%%%%%%%%%%%%%%%%%%%%%%%%%%%%%%%%%%%%%%%%%%%%%%%%%%%%%%%%%%%%%%%%%%%  
   %%%%%%%%%%%%%%%%%%%%%%%%%%DIMOSTRAZIONE LEMMA 1 LIMITATEZZA OPERATORE DI COEFFICIENTI%%%%%%%%%%%%%%%%%%%%%%%%%%%%%%%%%%%%%%%%%%%          
   %%%%%%%%%%%%%%%%%%%%%%%%%%%%%%%%%%%%%%%%%%%%%%%%%%%%%%%%%%%%%%%%%%%%%%%%%%%%%%%%%%%%%%%%%%%%%%%%%%%%%%%%%%%%%%%%%%%%%%%%%%%%%%%%  
    \begin{proof}
      A well known result (see e.g.~\cite{CH03}[Thm 7.2.3 p.143]) states that the Bessel property \eqref{eq:bessel}
      of $\ptg{ \widetilde{\phi_{j,k,\lambda}}(t)}_{\lambda \in \nu\Z}$, for fixed $j$,
      is equivalent to the following condition
      \[
        \Xi_{j,k}(\gamma)=
        \sum_{m \in \Z} \abs{ \F\pt{\widetilde{\phi_{j,k,0}}}
        \pt{ \pt{\gamma-m} \frac{\bap}{\nu}}}^2\leq C_\nu \frac{\nu}{\bap}, \quad a.e.\; \gamma \in [0,1].
      \]
      By definition
      \begin{align*}
        \Xi_{j,k}(\gamma)&= \frac{1}{\bap}\sum_{m\in \Z} \abs{\frac{2^{js}}{(1+\abs{\pt{\gamma-m} \frac{\bap}{\nu}})^s} \Phip\pt{\pt{\gamma-m} \frac{\bap}{\nu}}}^2.
      \end{align*}
      Using the hypothesis \eqref{eq_fcnd} and relations \eqref{eq:bd_ups1}, \eqref{eq:bd_ups2} 
      we can write
      \begin{align}
      \nonumber
        \Xi_{j,k}(\gamma)& \lesssim \frac{1}{\bap}\pt{
        1+ \sum_{|m|>1} \frac{2^{j}}{ \pt{1+ d\pt{\pt{\gamma-m} \frac{\bap}{\nu}, I_{j,k}}}^{2\alpha}}}\\
        \nonumber
        &\lesssim \frac{1}{\bap}\pt{
        1+ \sum_{|m|> 1} \frac{2^{j}}{  \pt{(|m|-1) \frac{\bap}{\nu} }^{2\alpha}}}\\
        \label{eq:ultimalim}
        &\lesssim \frac{1}{\bap}\pt{1+ \frac{\nu^{2\alpha}}{2^{j(2\alpha-1)}}
         \sum_{|m|> 1} \frac{1}{  \pt{(|m|-1) }^{2\alpha}}}, \quad \mbox{a.e. }\; \gamma\in [0,1].
      \end{align}
      By hypothesis $2\alpha>1$, therefore the sum in \eqref{eq:ultimalim} is convergent, 
      moreover it is uniformly bounded with respect to $j$. 
   %    To conclude the proof, set $C_\nu$ in function of \eqref{eq:ultimalim}.
   \end{proof}   
 
    \begin{lem}
       \label{lem:2Bessel}
      Let $\Phip$, $\widetilde{\phi_{j,k,\lambda}}$ as in Lemma \ref{lem:1Bessel} and $E_{j,k}$ as in \eqref{eq:epnu}.    
      If $\nu\in (0,1]$ and 
      \begin{align*}
      \supp \widehat{f} \cap E_{j,k} =\emptyset,
      \end{align*} 
      then
      \begin{align}
        \label{eq:bessel2}
        \sum_{\lambda \in \nu\Z} \abs{\Lprod{\widetilde{\phi_{j,k,\lambda}}} {f}}^2\leq \frac{C_{\nu}}{2^{j(2\alpha -1)}} \norm{f}^2_{L^2(\R)},
      \end{align}
         with $C_{\nu}$ as in Lemma \ref{lem:1Bessel}, therefore independent on $j,k$.
    \end{lem}
   \begin{proof}
     Since $\hat{f}(\omega)=0$ if $\omega\in E_{j,k}$
     \begin{align*}
       \sum_{\lambda \in \nu\Z} \abs{\Lprod{\widetilde{\phi_{j,k,\lambda}}} {f}}^2%&=
      % \sum_{\lambda \in \nu\Z} \abs{\Lprod{ \F\pt{\widetilde{\phi_{j,k,\lambda}}}} {\F(f)}}^2\\
       &=\sum_{\lambda \in \nu\Z} \abs{\Lprod{ \chi_{\R\setminus E_{j,k} }\F\pt{\widetilde{\phi_{j,k,\lambda}}}} {\F(f)}}^2.
     \end{align*}
     Therefore, using the same property of Lemma \ref{lem:1Bessel},
     \eqref{eq:bessel2} is equivalent to prove that
     \begin{align}
       \label{eq:final}
       \frac{1}{\bap}\sum_{m\in \Z} \abs{ \chi_{\R \setminus E_{j,k}} 
       \pt{m_{\gamma,j,\nu}} \frac{2^{js}}{\pt{1+\abs{m_{\gamma,j,\nu}}}^s} \Phip\pt{m_{\gamma,j,\nu}}}^2\leq \frac{1}{2^{j(2\alpha)}},
     \end{align}
      where $m_{\gamma,j,\nu} = \pt{\gamma-m} \frac{\bap}{\nu}$.
      Since $\nu\leq 1 $, 
      for each $j,k$, there exist a finite number of consecutive indexes $m$ such that $\pt{\gamma-m} \frac{\bap}{\nu} \in  E_{j,k}$ for some $\gamma\in [0,1]$. We set 
      \[
       M_{j,\nu} = \ptg{m \in \Z :\: \exists\; \gamma \in [0,1] \textrm{ such that } \pt{\gamma-m} \frac{\bap}{\nu} \in  E_{j,k}}.
       \]
     We notice that $M_{j,\nu}$ is uniformly bounded with respect to $j$, by the properties of the partitioning 
     and by the definition of $E_{j,k}$.
     %is finite set of integer numbers around $m=0$.
    \\
    If $m\in M_{j,\nu}$ and  $m_{\gamma,j,\nu}\in E_{j,k}$, then
     \[
      \abs{ \chi_{\R \setminus E_{j,k}} \pt{m_{\gamma,j,\nu}} \frac{2^{js}}
      {\pt{1+\abs{ m_{\gamma,j,\nu}}}^s} \Phip\pt{m_{\gamma,j,\nu}}}^2 = 0 .
     \]
     Otherwise 
    $\chi_{\R \setminus E_{j,k}} \pt{m_{\gamma,j,\nu}}=1$ and, using Lemma~\ref{L_uplow},
      \[
        \abs{  \frac{2^{js}}{\pt{1+\abs{m_{\gamma,j,\nu} } }^s} \Phip\pt{m_{\gamma,j,\nu}}}^2 \lesssim
        \abs{\frac{2^{j/2}}{\pt{1+\bap}^{\alpha}} }^2 \lesssim 2^{j(1-2\alpha)} .
      \]
   Hence, \eqref{eq:final} is bounded by
   \begin{align}\label{eq_mpnu}
   \abs{M_{j,\nu}}2^{-2\alpha j} \!+\! \frac{1}{\bap}\sum_{m\notin M_{j,\nu}} \abs{ \chi_{\R \setminus E_{j,k}} \pt{m_{\gamma,j,\nu}} \frac{2^{js}}{\pt{1+\dist{m_{\gamma,j,\nu}}{I_{j,k}}}^s} \Phip\pt{m_{\gamma,j,\nu}}}^2.
   \end{align}
   The second term in the equation above may be bounded as follows
      \begin{align*}
         & \frac{1}{\bap}\sum_{m\notin M_{j,\nu}} \abs{ \chi_{\R \setminus E_{j,k}}
         \pt{m_{\gamma,j,\nu}} \frac{2^{js}}{\pt{1+\dist{m_{\gamma,j,\nu}}{I_{j,k}}}^s} 
         \Phip\pt{m_{\gamma,j,\nu}}}^2\\
         &\qquad\lesssim \frac{1}{\bap} \frac{\nu^{2\alpha} 2^j }{2^{j(2\alpha)}}
         \sum_{|m|\geq 2} \frac{1}{  \pt{(|m|-1) }^{2\alpha}}\\
         &\qquad\lesssim  \frac{\nu^{2\alpha} } {2^{j(2\alpha)}}
         \sum_{|m|\geq 2} \frac{1}{  \pt{(|m|-1) }^{2\alpha}}\\
         &\qquad\lesssim \frac{\nu^{2\alpha} } {2^{j(2\alpha)}}.
      \end{align*}
   Then the assertion follows as in Lemma \ref{lem:1Bessel}.
   \end{proof}
   %%%%%%%%%%%%%%%%%%%%%%%%%%%%%%%%%%%%%%%%%%%%%%%%%%%%%%%%%%%%%%%%%%%%%%%%%%%%%%%%%%%%%%%%%%%%%%%%%%%%%%%%%%%%%%
   %%%%%%%%%%%%%%%%%%%%%%%THEOREM BOUNDNESS%%%%%%%%%%%%%%%%%%%%%%%%%%%%%%%%%%%%%%%%%%%%%%%%%%%%%%%%%%%%%%%%%%%%%%
   %%%%%%%%%%%%%%%%%%%%%%%%%%%%%%%%%%%%%%%%%%%%%%%%%%%%%%%%%%%%%%%%%%%%%%%%%%%%%%%%%%%%%%%%%%%%%%%%%%%%%%%%%%%%%%
   \begin{thm}\label{thm:Bound}
       Let $s\geq 0$ and $\varphi$ so that condition \eqref{eq_fcnd} holds for $\alpha>\frac{1}{2}$, 
   %       Moreover,
   %       \[
   %       \widehat\varphi(\xi) \lesssim (1+|\omega|)^{-\alpha}. 
   %       \]
   %        \item[iii)] 
   %          \begin{align}
   %          \label{eq:con3}  
   %            |\Phip(\omega)|\lesssim 1, \quad \xi\in I_p.
   %          \end{align}
   %          
      then, there exists a positive constant $C$ such that 
      \begin{equation}
       \label{eq_coeff_ops}
       \sum_{\varb}2^{2js}\abs{\Lprod{f}{\phipkn}}^2 \leq C \norm{f}^2_s,
      \end{equation}
   where $C$ depends on $\nu$.   
      \end{thm}
   %%%%%%%%%%%%%%%%%%%%%%%%%%%%%%%%%%%%%%%%%%%%%%%%%%%%%%%%%%%%%%%%%%%%%%%%%%%%%%%%%%%%%%%%%%%%%%%%%%%%%%%%%%%%%%%%%%%%%%%%%%%%%%%%  
   %%%%%%%%%%%%%%%%%%%%%%%%%%DIMOSTRAZIONE LIMITATEZZA OPERATORE DI COEFFICIENTI%%%%%%%%%%%%%%%%%%%%%%%%%%%%%%%%%%%%%%%%%%%%%%%%%%% 
   %%%%%%%%%%%%%%%%%%%%%%%%%%%%%%%%%%%%%%%%%%%%%%%%%%%%%%%%%%%%%%%%%%%%%%%%%%%%%%%%%%%%%%%%%%%%%%%%%%%%%%%%%%%%%%%%%%%%%%%%%%%%%%%%  
     \begin{proof}
    
      For each index $j$ set
      \begin{align*}
        f_{j,k,1}(t)= \F^{-1}_{\omega\to t} \pt{\chi_{E_{j,k}(\omega)} \widehat{f}(\omega)},\quad
        f_{j,k,2}(t)= 1- f_{j,k,1}(t),
      \end{align*}
         where $E_{j,k}$ is as in \eqref{eq:epnu}.	
   Notice that
    \(
         E_{j,k} \cap E_{j',k}=\emptyset, \quad \mbox{if } |j-j'|\geq 2.
    \)
       By Plancherel Theorem, using the notations of Lemmata \ref{lem:1Bessel}, \ref{lem:2Bessel}, 
      we rewrite
      \begin{align}
        &\sum_{\varb}2^{2js}\abs{\Lprod{f}{\phipkn}}^2 =
        \sum_{\varb}2^{2js}\abs{\Lprod{\F (f)}{\F(\phipkn)}}^2\nonumber\\
        &=\sum_{\varb}\abs{\Lprod{\F (f) (1+|\omega|)^s}{ \frac{2^{js}}{(1+|\omega|)^s}\F(\phipkn)}}^2\nonumber\\
        &=\sum_{\varb}\abs{\Lprod{\F (f) (1+|\omega|)^s}{\frac{1}{2^{j/2}} e^{-\pii \frac{\lambda}{\bap}  (\cdot)}\Phip(\cdot) \frac{2^{js}}{(1+|\omega|)^s} } }^2\nonumber\\
      &=\sum_{\varb}\abs{\Lprod{\F (f) (1+|\omega|)^s}{  \F( \widetilde{\phi_{j,k,\lambda}) }}}^2\label{eq:thm3_last}.
   \end{align}
   Splitting $f = f_{j,k,1}+f_{j,k,2}$ for each $j$ we get
   \begin{align*}   
   &\sum_{\varb}2^{2js}\abs{\Lprod{f}{\phipkn}}^2 \\
      &\lesssim\pt{\sum_{\varb}\abs{\Lprod{\F (f_{j,k,1}) (1+|\omega|)^s}{  \F( \widetilde{\phi_{j,k,\lambda}) }} }^2
      \!+\! \sum_{\varb}\abs{\Lprod{\F (f_{j,k,2})(1+|\omega|)^s}{  \F( \widetilde{\phi_{j,k,\lambda}) }} }^2}.
      \end{align*}
  
     Notice that $f_{j,k,2}$ satisfies the hypothesis of Lemma \ref{lem:2Bessel}, therefore
     \begin{align*}
       \sum_{\varb} 2^{2js} \abs{\Lprod{f}{\phipkn}}^2 &\lesssim \frac{1}{\nu}    \sum_{j,k}\pt{\norm{f_{j,k,1})}_s^2  +\frac{{1}} {2^{j(2\alpha-1)} } \norm{\F\pt{f_{j,k,2}} (1+|\omega|)^s}^2}\\
   %\\
    %   &{}\qquad\lesssim 
   %\frac{1}{\nu}    \sum_{j,k}\pt{\norm{\F\pt{f_{j,k,1}} (1+|\omega|)^s}^2  +\frac{{1}} {2^{j(2\alpha-1)} } \norm{\F\pt{f_{j,k,2}} (1+|\omega|)^s}^2}\\
   \\
        &\lesssim \nu^{-1}\|f\|^2_s,
     \end{align*}
   as desired. Indeed, since the partition $\ptg{I_{j,k}, I_\bullet}$ is admissible then $E_{j,k}$
   have the (uniform) finite intersection property as well, then
   \[
   \sum_{j,k}\norm{f_{j,k,1} }_s^2 \leq \sum_{j,k}\norm{f}^2_{H^s(E_{j,k})}\lesssim \norm{f}^2_s,
   \]   
   and 
   \[
   \sum_{j,k}\frac{{1}} {2^{j(2\alpha-1)} } \norm{\F\pt{f_{j,k,2}} (1+|\omega|)^s}^2 \leq  \norm{f}^2_s\sum_{j,k}\frac{{1}} {2^{j(2\alpha-1)} }\lesssim  \norm{f}^2_s.
   \]
   \end{proof}
   \begin{cor}
   The analysis operator 
   \begin{align*}
   C:  L^2(\R) &\longrightarrow  \ell^2(\Gamma)\\
   	f &\longmapsto \ptg{\Lprod{f}{\phi_{j,k,\lambda}}}_{\varb}
   \end{align*}
   is continuous. Hence, the same is true for the frame operator $S=C^{*}C$.
   \end{cor}

   \subsection{Lower bound}
   Using the hypothesis on the window functions, we show that there exists a (uniform) lower
   bound for the $H^s$-norm. 
   \begin{lem}
     \label{lem:low}
     Let $s\geq0$ and $\varpb,\varphi$ be a system of functions such that
   \[
   \abs{\Phi_{\bullet}\pt{\omega}} \geq a, \:\omega \in I_{\bullet}
   \]  
   and  for all $\omega \in \R\backslash I_{\bullet}$ there exists $\bar{j},\bar{k}$ such that $\omega \in I_{\bar{j},\bar{k}}$ and
     \[
   	   |\Phi_{\bar{j},\bar{k}}(\omega)|\geq a
     \]
     and the constant $a$ does not depend on $\bar{j},\bar{k}$.
     Then
     \begin{equation}
         \label{loL2}
              \frac{\abs{\Phi_{\bullet}(\omega)}^2}{\pt{1+|\omega|}^{2s}}+ \sum_{j,k\in\Gamma} \frac{2^{2js}}{\pt{1+|\omega|}^{2s}} |\Phip(\omega)|^2 \geq a^2, \quad \mbox{a.e. }\omega\in \R\;.
     \end{equation}
   \end{lem}
      \begin{proof}
   For $s=0$, the statement is trivial while for general $s$, notice that 
   \[
   (1+|\omega|)\sim 2^j, \qquad \omega \in I_{j,k},
   \]  
   while if $\omega \in I_{\bullet}$, then $(1+|\omega|)\sim 1$.
      \end{proof}
   \begin{rem}
    Inequality \eqref{loL2} could be used as hypothesis on the window functions weaker then ours. 
    Since it is quite cumbersome to be checked, we prefer to work with a more transparent assumptions.
   \end{rem}

   We can state the main result on the lower bound for the $H^s$ frame.
   % 
   %Theorem LIMITATEZZA da sotto 
   % 
   \begin{thm}\label{thm_down_op}
   Let $s\geq0$ and $\varpb,\varphi$ be $s$-admissible. Then, there exist $\nu_0>0$ and $C>0$ such that for every $\nu\in (0,\nu_0)$, we have
      \begin{equation}\label{eq_coeff_ops_low}
   \sum_{\lambda\in\nu\Z}|\Lprod {T_{\lambda} \varpb}{f}|^2  + \sum_{\varb}2^{2js}\abs{\Lprod{f}{\phipkn}}^2 \geq C\norm{f}^2_s,
      \end{equation}
   where the constant $C$ depends on $\nu$.
   \end{thm}
   \begin{proof}
   Consider the (modified) frame operator
   \[
   S^{s}f(x) = S^s_{\varpb} f(x) + S_{\varphi}^s f(x),
   \]
   where
   \begin{align*}
   S^s_{\varpb} f(x) {}= \sum_{\lambda}\Lprod{f}{\varphi_{\bullet,\lambda}}\varphi_{\bullet,\lambda}\,, \qquad 
   S_{\varphi}^s f(x) {}= \sum_{j,k,\lambda}2^{2js}\Lprod{f}{\phipkn}\phipkn(x).
   \end{align*}  
   We use Daubechies-like  (or Walnut-like) representation formula, see \cite[Lemma 4.6]{BB15}, to rewrite 
     \begin{align}
   \label{eq_wal3}
       \Lprod{S^s f(x)}{f(x)} &= 
       \langle \sum_{\sigma,m\in\Z}4^{\sigma s} T_{{\frac{m}{\nu}\bas}}\pt{\overline{\Phi_{\sigma} \widehat{f}}} \Phi_{\sigma}, \widehat{f}\rangle_{L^2(\R)},
   \end{align}
   where $\sigma \in \ptg{\bullet, \pt{j,k}}$ and, with an abuse of notation, we set
   \[
   2^{\sigma} = 2^j, \textrm{ for } \sigma = \pt{j,k}, \qquad 2^{\sigma} = 1, \textrm{ for } \sigma = \bullet.
   \]
   % Using the s-admissibility - cf.~\eqref{eq:con2} - we can reinterpret equations~\eqref{eq:bd_ups1}-\eqref{eq:bd_ups2} as
   % \begin{equation}
   % \begin{split}
   %     \label{eq:bd_ups3}    
   %      \frac{2^{js}}{(1+|\omega|)^{s}} |\Phi_{\sigma}(\omega)| &\lesssim \frac{2^{j/2}}{\pt{1+d(\omega, I_{\sigma})}^{\alpha}}  \quad \mbox{a.e. }\omega \notin I_{\sigma}
   %      \\
   %      \frac{2^{js}}{(1+|\omega|)^{s}} |\Phi_{\sigma}(\omega)| &\lesssim 1  \quad \mbox{a.e. }\omega \in I_\sigma.      
   %   \end{split}
   % \end{equation}
   First, for $m=0$ we apply~\eqref{loL2} and obtain
     \begin{align*}
   % &\Lprod{S^s f(x)}{f(x)} \\
   &    \langle \sum_{\sigma}4^{\sigma s}\frac{\abs{\Phi_{\sigma}}^2}{(1+|\omega|)^{2s}} (1+|\omega|)^{s}\widehat{f}, (1+|\omega|)^{s}\widehat{f}\rangle_{L^2(\R)}\\
   &= \langle\ptq{ \frac{|\Phi_{\bullet}|^2}{(1+|\omega|)^{2s}} + \sum_{j,k} 4^{\sigma s}\frac{|\Phip|^2}{(1+|\omega|)^{2s}}} (1+|\omega|)^s \hat{f}, (1+|\omega|)^s \hat{f}\rangle_{L^2(\R)}\geq a^2 \norm{f}_s^2.     
     \end{align*}  

   Hence, 
    \begin{align}
       \Lprod{S^s f(x)}{f(x)} \geq a^2 \norm{f}_s^2 +\langle \sum_{\sigma}4^{\sigma s}\sum_{m\neq 0}T_{\frac{m}{\nu}\bas}\pt{\overline{\Phi_{\sigma} \widehat{f}}} \Phi_{\sigma}, \widehat{f}\rangle_{L^2(\R)} .
   \end{align}
       We study the last term of the above sum  by splitting the different components. Precisely,
   \begin{align}
   \nonumber
   &\langle \sum_{\sigma}4^{\sigma s}
   \sum_{m\neq0} T_{\frac{m}{\nu}\bas}\pt{\overline{\Phi_{\sigma} \widehat{f}}} \Phi_{\sigma}, 
   \widehat{f}\rangle_{L^2(\R)}=
   R_{1,1}+R_{1,2}+R_{2,1}+R_{2,2}
   \end{align}
   where
   \begin{align}
   \label{eq:f1f1}
   R_{1,1}&=     \langle \sum_{\sigma} 4^{\sigma s}\sum_{m\neq0} T_{\frac{m}{\nu}\bas}\pt{\overline{\Phi_{\sigma} \widehat{f_{\sigma,1}}}} \Phi_{\sigma} , \widehat{f_{\sigma,1}}\rangle_{L^2(\R)}  	
   \\
         \label{eq:f2f1}
         R_{2,1}&=\langle \sum_{\sigma}4^{\sigma s}\sum_{m\neq0} T_{\frac{m}{\nu}\bas}\pt{\overline{\Phi_{\sigma} \widehat{f_{\sigma,2}}}} \Phi_{\sigma}, \widehat{f_{\sigma,1}}\rangle_{L^2(\R)}
   \\
         \label{eq:f1f2}
         R_{1,2}&=\langle\sum_{\sigma}4^{\sigma s}\sum_{m\neq0} T_{\frac{m}{\nu}\bas}\pt{\overline{\Phi_{\sigma} \widehat{f_{\sigma,1}}}} \Phi_{\sigma}, \widehat{f_{\sigma,2}}\rangle_{L^2(\R)}
         \\
         \label{eq:f2f2}
               R_{2,2}&=\langle \sum_{\sigma}4^{\sigma s}\sum_{m\neq0} T_{\frac{m}{\nu}\bas}\pt{\overline{\Phi_{\sigma} \widehat{f_{\sigma,2}}}} \Phi_{\sigma}, \widehat{f_{\sigma,2}}\rangle_{L^2(\R)}, 
     \end{align}
   and
   \[
   f  = f_{\sigma,1} + f_{\sigma,2}, \textrm{ and } \widehat{f}_{\sigma,1} = \widehat{f}\chi_{E_{\sigma}},
   \]
    $E_{\sigma}$ is defined in \eqref{eq:epnu}.
   We want to show that the terms in (\ref{eq:f1f1}-\ref{eq:f1f2}) go to zero as $\nu$ does. Precisely, we show that there exists $\nu_0>0$ such that for each $\nu<\nu_0$ one has
   \begin{equation}
   \label{eq:tozero}
   \begin{split}
    \frac{a}{2}\norm{f}^2_s   &\geq R_{1,1}+R_{2,1}+R_{1,2}   \,.
   \end{split}
   \end{equation}
   Also, we know  that $R_{2,2}\geq 0$.
   Indeed, we can rewrite \eqref{eq:f2f2} as
   \[
   \sum_{\sigma} 2^{2s\sigma}\abs{\Lprod{f_{\sigma,2}}{\varphi_{\sigma,\lambda}}_{L^2(\R)}}^2,
   \]
   which is quadratic hence positive, as desired.% Then (\ref{eq:tozero}-\ref{eq:f2f2}) yields \eqref{eq_coeff_ops_low}.
   \medskip\\
   \textbf{Step 1} Show that \eqref{eq:f1f1} is identically zero.\medskip\\
   Equation \eqref{eq:f1f1} vanishes for all  $m\neq0$, since the support of $\widehat{f}_{\sigma,1}$ is compact.\\
   Indeed, assuming $\nu<\frac{1}{2}$, if $\omega\in E_{\sigma}$, then 
   \[
   \omega - \frac{m}{\nu}\bas \notin E_{\sigma}, \quad m\in\Z,
   \] 
   since $|E_{\sigma}| \leq  2^{\sigma+1}$.\\
   \textbf{Step 2.} Show that the term in \eqref{eq:f2f1} goes to zero as $\nu$ does.\medskip\\
   We rearrange the sum as
   \begin{align*}
   &\sum_{\sigma} 4^{s\sigma}\sum_{m\neq 0} \Lprod{\overline{\widehat{f_{\sigma,2}}}\pt{\cdot -m \frac{\bas}{\nu}}\Phi_{\sigma}}{{\Phi_{\sigma}}\pt{\cdot -m \frac{\bas}{\nu}}\widehat{f_{\sigma,1}}}_{L^2(\R)}.
   \end{align*} 
   Multiply and dividing 
   by $(1+\abs{\omega})^{s},\pt{1+\abs{\omega-m \frac{\bas}{\nu}}}^{s} $,
   since $f_{\sigma,1}$ vanishes outside $E_\sigma$, we are led to
   \begin{align*}
   % &\sum_{\sigma}4^{\sigma s}\sum_{m\neq 0} \int_{\R}\pt{1+\abs{\omega -m \frac{\bas}{\nu}}}^s\overline{\widehat{f_{\sigma,2}}}\pt{\omega -m \frac{\bas}{\nu}}\frac{\Phi_{\sigma}(\omega)}{(1+|\omega|)^{s}}
   % \\
   %&\qquad\qquad\qquad\frac{\overline{\Phi_{\sigma}}\pt{\omega -m \frac{\bas}{\nu}}}{\pt{1+\abs{\omega -m \frac{\bas}{\nu}}}^s}(1+|\omega|)^{s}\overline{\widehat{f_{\sigma,1}(\omega)}}\ud\omega
   %\\
    &\sum_{\sigma}4^{\sigma s}\sum_{m\neq 0}\int_{E_{\sigma}}\pt{1+\abs{\omega -m \frac{\bas}{\nu}}}^s\overline{\widehat{f_{\sigma,2}}}\pt{\omega -m \frac{\bas}{\nu}}\frac{\Phi_{\sigma}(\omega)}{(1+|\omega|)^{s}}
   \\
   &\qquad\qquad\frac{\overline{\Phi_{\sigma}}\pt{\omega -m \frac{\bas}{\nu}}}{\pt{1+\abs{\omega -m \frac{\bas}{\nu}}}^s}(1+|\omega|)^{s}\overline{\widehat{f_{\sigma,1}(\omega)}}\ud\omega.
   \end{align*}
   Then by Cauchy-Schwartz Inequality 
   \begin{align}
   \label{eq:bd_sum}
   \langle \sum_{\sigma} 4^{\sigma s}\sum_{m\neq 0} T_{m \frac{\bas}{\nu}}\pt{\overline{\Phi_{\sigma} \widehat{f_{\sigma,2}}}} \Phi_{\sigma}, \widehat{f_{\sigma,1}}\rangle_{L^2(\R)} &\leq 	\sum_{\sigma}\sum_{m\neq 0} c^{1/2}_{\sigma,m}d^{1/2}_{\sigma,m},
   \end{align}
   where
   \begin{align*}
   c_{\sigma,m} &= \int_{E_{\sigma}}\abs{\overline{\widehat{f_{\sigma,2}}}\pt{\omega -m \frac{\bas}{\nu}}}^2\pt{1+\abs{\omega -m \frac{\bas}{\nu}}}^{2s} 4^{\sigma s}\frac{\abs{\Phi_{\sigma}(\omega)}^2}{(1+|\omega|)^{2s}}\ud\omega
   \\
   &\lesssim\int_{E_{\sigma}}\abs{\overline{\widehat{f_{\sigma,2}}}\pt{\omega -m \frac{\bas}{\nu}}}^2\pt{1+\abs{\omega -m \frac{\bas}{\nu}}}^{2s}\ud\omega
   \end{align*}
    and
   \begin{align*}
   d_{\sigma,m} &= \int_{E_{\sigma}}\abs{\overline{\widehat{f_{\sigma,1}}}\pt{\omega}}^2\pt{1+\abs{\omega}}^{2s} 4^{\sigma s}\frac{\abs{\Phi_{\sigma}\pt{\omega -m \frac{\bas}{\nu}}}^2}{\pt{1+\abs{\omega -m \frac{\bas}{\nu}}}^{2s}}\ud\omega
   \\
   &\leq\int_{E_{\sigma}}\abs{\overline{\widehat{f_{\sigma,1}}}\pt{\omega}}^2\pt{1+\abs{\omega}}^{2s} \frac{\bas}{\pt{1+\dist{\omega -m \frac{\bas}{\nu}}{I_{\sigma}}}^{2\alpha}}\ud\omega
   \\
   &\leq \sup_{\omega \in E_{\sigma}}\ptg{\frac{\bas}{\pt{1+\dist{\omega -m \frac{\bas}{\nu}}{I_{\sigma}}}^{2\alpha}}}\int_{E_{\sigma}}\abs{\overline{\widehat{f_{\sigma,1}}}\pt{\omega}}^2\pt{1+\abs{\omega}}^{2s}\ud\omega.
   \end{align*}
   Notice that we have used inequalities \eqref{eq:bd_ups1}, \eqref{eq:bd_ups2}.
   Since
   \[
   \dist{\omega -m \frac{\bas}{\nu}}{I_{\sigma}} \geq \frac{\bas}{\nu}\pt{|m|-2\nu},\quad \omega \in E_{\sigma},
   \]
   if $\nu<\frac{1}{2}$, then $\pt{|m|-2\nu}>0$, for any $m\neq 0$.
   Therefore, 
   \[
   d_{\sigma,m} \leq \nu^{2\alpha}2^{\sigma(1-2\alpha)}\frac{1}{\pt{|m|-2\nu}^{2\alpha}} \norm{f_{\sigma,1}}^2_s.
   \]
   % The coefficients $c_{\sigma,m}$ and $d_{\sigma,m}$ are both $\ell^2$-summable w.r.t. of $m$, precisely
   %   \[
   %       \norm{c_{\sigma,\cdot}
   % }_{\ell^2(\Z)}=\pt{\sum_{m\neq 0} c^{(1)}_{\sigma,m}}^{1/2}\leq \norm{f}_{s}\]
   % \[
   %       \norm{d_{\sigma,\cdot}
   % }_{\ell^2(\Z)}=\pt{\sum_{m\neq 0} |d_{\sigma,m}|}^{1/2}\leq  \nu^{\alpha}\norm{f_{\sigma,1}}_s.
   % \]
   \noindent
   Hence, using the properties of the scalar product and the norm in the space $\ell^2(\Gamma)$ 
   with parameters $(\sigma,m)$, we can write
   %     \[
   % 	\sum_{\sigma, m\neq 0} c^{1/2}_{\sigma, m}d^{1/2}_{\sigma, m} 
   % 	\lesssim  
   % 	\sum_{\sigma, m\neq 0 }\norm{c^{}_{\sigma,\cdot}
   % 	}_{\ell^2(\Z)}\norm{d_{\sigma,\cdot}
   % 	}_{\ell^2(\Z)} \lesssim \norm{f}_s\nu^{\alpha}\sum_{\sigma}\norm{f_{\sigma,1}}_s.
   %     \]

   \begin{align}
   	\nonumber
   	\sum_{\sigma, m\neq 0} c^{1/2}_{\sigma, m}d^{1/2}_{\sigma, m} 
   	&\lesssim  
   	\sum_{\sigma, m\neq 0 }c^{1/2}_{\sigma, m}  \nu^\alpha 2^{\sigma/2(1-2\alpha)} \frac{1}{\pt{|m|-2\nu}^\alpha}\norm{f_{\sigma,1}}_s\\
   	\label{eq:quasifin}
   	&\lesssim
   	\nu^{\alpha} \norm{c^{1/2}_{\sigma, m} 
   	2^{\sigma/2(1-2\alpha)}}_{\ell^2} \norm{\frac{1}{\pt{|m|-2\nu}^\alpha} \norm{f_{\sigma,1}}_s}_{\ell^2}\,.
   \end{align}

   Since $\ptg{I_{j,k}, I_{\bullet}}$ is an admissible partition,
   $\ptg{E_\sigma}$ is as well a partition of $\R$ with the finite intersection property and since 
   $\alpha>\frac{1}{2}$
   \begin{align}
      \nonumber
      \norm{\frac{1}{\pt{|m|-2\nu}^\alpha} \norm{f_{\sigma,1}}_s}^2_{\ell^2} 
      %&=
      %\sum_{\sigma,m\neq 0} \frac{1}{\pt{|m|-2\nu}^{2\alpha}} \norm{f_{\sigma,1}}^2_s\\
      %\nonumber
      &\leq  \pt{\sum_{m\neq 0}\frac{1}{\pt{|m|-2\nu}^{2\alpha}}} \sum_{\sigma}\norm{f_{\sigma,1}}^2_s\\
      \label{eq:primapart}
      &\lesssim \norm{f}^2_s\,.
   \end{align}
   Moreover, 
   \begin{align}
     \nonumber
     \norm{c^{1/2}_{\sigma, m} 2^{\sigma/2(1-2\alpha)}}^2_{\ell^2}&=\sum_{\sigma, m\neq 0}
     c_{\sigma, m} 2^{\sigma(1-2\alpha)}\\
     \nonumber
     &=
     \sum_{\sigma} \pt{ 2^{\sigma(1-2\alpha)} \pt{\sum_{m\neq 0} c_{\sigma,m}}}\\
     \label{eq:secondapart}
     &\lesssim
     \norm{f}_s^2 .
   \end{align}
   Combining \eqref{eq:primapart}, \eqref{eq:secondapart} with \eqref{eq:quasifin} we obtain the assertion.
   \medskip\\

   \textbf{Step 3} Show that the term in \eqref{eq:f1f2} goes to zero as $\nu$ does.\medskip\\
   This follows from the previous step using a change of variable of integration. 
   \end{proof}

   \subsection{proof of theorem~\ref{thm:fram}}
   \begin{proof}
   	The assertion follows combining Theorems \ref{thm:Bound} and \ref{thm_down_op}.
   \end{proof}

\subsection{$H^s$-seminorm discretization}
 In Theorem \ref{thm:fram} we proved that $F(\varpb,\varphi,\Gamma)$ is a frame that describes the $H^s$-norm, provided the parameter $\nu$ is small enough.
 Going through the proof it is clear that, under the hypothesis of the previous section,
 it is not possible to describe the $H^s$-seminorm as well. The problem arises, near the point $\omega=0$,
 in the frequency domain. The main reason is that 
 the covering \eqref{eq:setspl} is too coarse near the origin, in particular,
 the set $I_\bullet$ is too large. 
 Therefore, as in the wavelets case, we need vanishing moments in the origin to  detect high order singularities,
 see \cite{mallat1999wavelet,ME92}. In Figure~\ref{Fig:Vanish}, we show an example of window with such properties.\medskip\\
 \begin{figure}
    \centering
    \begin{subfigure}[b]{0.4\textwidth}
\includegraphics[width=\textwidth]{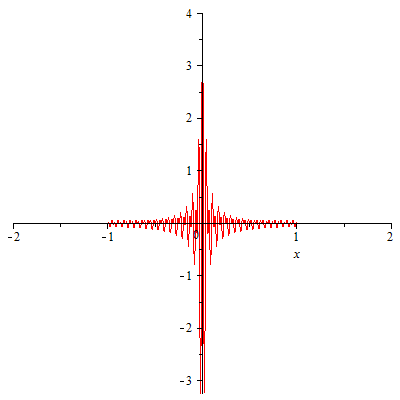}
        \caption{Time domain}
    \end{subfigure}
    \begin{subfigure}[b]{0.4\textwidth}
        \includegraphics[width=\textwidth]{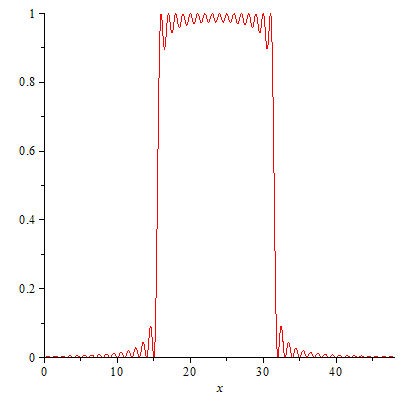}
        \caption{Frequency domain}
    \end{subfigure}
\caption{Frame element $\phipkn(t)$ in both time and frequency. Here, $\phi = \chi_{\ptq{-\frac{1}{2},\frac{1}{2}}}\ast\chi_{\ptq{-\frac{1}{2},\frac{1}{2}}}$, and $j=4$.}\label{Fig:Vanish}
\end{figure}
We introduce the dilation operator $\Dil_a: L^2(\R^d)\to L^2(\R^d)$
 \begin{align*}
	 f(x)\mapsto (\Dil_a f)(x)= a^{-d/2} f\pt{ \frac{ x}{a}}, \quad a\in \R\setminus \ptg{0}.
 \end{align*}
 which is unitary on $L^2(\R^d)$.

Let $\ptg{I_{j,k}}_{j,k\in\Gamma}$ be an admissible partition  in the sense of Definition~\ref{def:adm},  
we consider a new set of indexes
 \[
   \Delta=\ptg{(j, k , \lambda) \mid j\in \Z, k\in K, \lambda \in \nu\Z^d},
 \]
with $|K|<+\infty$ and the sets 
 \begin{align}
   &I_{-j,k}= \ptg{x\in \R^d \mid 2^{2j} x \in I_{j,k} }, & j\in \N \setminus\ptg{0}, \quad k \in K.\nonumber
 \end{align}
\begin{rem}
	These new sets are just contractions of the original ones. We define these sets in order to refine the analysis on $I_\bullet$; clearly the union of these covers $I_{\bullet}$ 
	entirely, except the point $\ptg{0}$ which has zero measure.
\end{rem}
 
 For each $\varphi$, we define the new frame
 \begin{align}
   \nonumber
   \varphi_{j,k,\lambda}(t)&=\varphi_{j,k,\lambda}(t), \quad j\in \N, k\in K,\\
   \label{eq:framesemi}
   \varphi_{-j,k,\lambda}(t)&= T_{\lambda \bap}
                 \Dil_{2^{ 2j}}\varphi_{j,k,0}(t)= \Dil_{2^{2j}} T_{\frac{\lambda}{\bap}}\varphi_{j,k,0}(t) , \quad j \in \N\setminus \ptg{0}, \; k\in K.
 \end{align}
The natural expression of
$\Phi_{-j,k}(\omega)$ for $j\in \N \setminus \ptg{0}$ is:
 \[
   \Phi_{-j,k}(\omega)=  \Phi_{j,k}(2^{2j}\omega).
 \]
 
 We introduce now the admissibility criteria for the seminorm characterization.

\begin{Def}[$s$-admissible for the Sobolev seminorm]
  \label{def-sadmsemi}
Let $s\geq0$, we say that $\varphi$ is $s$-admissible for the Sobolev seminorm
with respect to the partition $\ptg{I_{j,k}}$ if

    \begin{itemize}
      \item[i)]there exists $\alpha>\frac{d}{2}$ such that for all $j\in \N $
	\begin{align}
        \label{eq:con2semi}
        |\Phip(\omega)|\lesssim \left\{
			    \begin{array}{ll}
			      \frac{2^{jd/2} \min\pt{1,|\omega|}^s } {\pt{1+d(\omega, I_{j,k})}^{\alpha+s}},& \quad \omega\neq I_{j,k}\\
			      1 &\quad \omega\in \R^d\,.
			    \end{array}
			   \right.  
       \end{align}
       %the function $\Phi_{j,k}(\omega)$ is defined as in \eqref{eq:G}.
   \item[ii)] There exists a constant $a>0$ such that, for a.e. $\omega\in\R^d$ there exists $\bar{j}\in 
   \Z$ and $\bar{k}\in K$ such that $\omega\in I_{\bar{j},\bar{k}}$ and
  \[
	   |\Phi_{\bar{j},\bar{k}}(\omega)|\geq a.
  \]
  The constant $a$ does not depend on $\bar{j},\bar{k}$.
    \end{itemize}	
\end{Def}
\begin{rem}
The requirement~\eqref{eq:con2semi} means that the frequency windows 
$\Phi_{j,k}$ needs to decay at infinity and to have a polynomial behavior in the origin. 
This property is necessary to avoid accumulation near the zero-frequency. 

As we noticed in Remark \ref{rem_decay_finite} the $\alpha+s$-decay properties can be weekend and 
assumed to hold for $j>j_0$, and to be just of type $\alpha$ for $j\leq j_0$.
\end{rem}

%TODO : aggiungere remark j0 -> infty
   
   We state the main Theorem.
 \begin{thm}
 \label{thm:mainsemi}
  For all $s\geq0$, if $\varphi$ is $s$-admissible for the Sobolev seminorm - cf. Definition \ref{def-sadmsemi},
    then there exits $\nu_0$  such that for all $\nu\in (0, \nu_0)$ the system of functions 
    \[
      \ptg{\varphi_{j,k,\lambda}\mid j, k, \lambda \in \Delta}
    \]
    is a frame representing the $H^s$-seminorm. That is, there exist constant $A, B>0$ such that
   \begin{equation}\label{eq_coeff_ops_low_s}
      A \abs{f}^2_s\leq \sum_{j,k,\lambda \in \Delta} 2^{2js}\abs{\Lprod{f}{\phipkn}}^2 \leq B \abs{f}^2_s.
   \end{equation}
\end{thm}

 The proof Theorem \ref{thm:mainsemi} is essentially divided in the proof of the upper bound and of the lower bound,
 as Theorem \ref{thm:fram}.
 For the upper bound we need two Lemmata similar to Lemma \ref{lem:1Bessel} and Lemma \ref{lem:2Bessel},
 the starting point is an extension of Lemma \ref{L_uplow}.
 \subsection{Preparatory results}
 As we did for the Sobolev norm, we develop the proof in dimension $d=1$ and we start with the upper bound.
 
 \begin{lem}
   \label{L_uplowsemi}
    Let $s\geq 0$ and $\varphi$ be a function such that for all $j \in \N$
    \begin{align}
       \label{eq_fcndsemi}
       |\Phi_{j,k} (\omega)|\lesssim \left\{
			    \begin{array}{ll}
			      \frac{ \min (1,|\omega| )^s 2^{j/2} }{\pt{1+d(\omega, I_j)}^{\alpha+s}},& \quad \omega\notin I_{j,k}\\
			      1 &\quad \omega \in \R,
			    \end{array}
			   \right.
    \end{align}
    for some $\alpha>\frac{1}{2}$.
%       Moreover,
%       \[
%       \widehat\varphi(\xi) \lesssim (1+|\xi|)^{-\alpha}. 
%       \]
%        \item[iii)] 
%          \begin{align}
%          \label{eq:con3}  
%            |\Phi_p(\xi)|\lesssim 1, \quad \xi\in I_p.
%          \end{align}
%          
     Then
     \begin{align}
	\label{eq:bd_ups1semi}    
	\frac{2^{j s}}{|\omega|^{s}} |\Phi_{j,k}(\omega)| &\lesssim \frac{2^{j/2}}{\pt{1+d(\omega, I_{j,k})}^{\alpha}}  &\quad \mbox{a.e. }\omega \notin I_{j,k} 
	\\
	\label{eq:bd_ups2semi}
	\frac{2^{js}}{|\omega|^{s}} |\Phi_{j,k}(\omega)| &\lesssim 1  &\quad \mbox{a.e. }\omega \in I_{j,k},      
     \end{align}
     for all $j \in \N$.
 \end{lem}
 
  \begin{proof}
 
   The proof of \eqref{eq:bd_ups2semi} is analogous to the proof of \eqref{eq:bd_ups2} of Lemma \ref{L_uplow}.
    Inequality \eqref{eq:bd_ups1semi} is first shown  if $|\omega|\leq 1$. 
    By triangular inequality, \eqref{eq_fcndsemi} implies
   \begin{align*}
     \frac{2^{js}}{|\omega|^s} \abs{\Phi_{j,k}(\omega)}&\lesssim 
     \frac{2^{js}}{|\omega|^s}\frac{2^{j/2} \min\ptg{1, |\omega|}^s} {\pt{1+d(\omega, I_{j,k})}^{\alpha+s}}
     %\lesssim 
     %\frac{2^{js}}{\pt{1+d(\omega,I_{j,k})}^s} \frac{2^{j/2}}{\pt{1+d(\omega,I_{j,k})}^{\alpha}}\\
     %& \lesssim 
     %\frac{2^{js}}{(1+2^j)^s }  \frac{2^{j/2}}{\pt{1+d(\omega,I_{j,k})}^{\alpha}} 
     \lesssim 
     \frac{2^{j/2}}{\pt{1+d(\omega,I_{j,k})}^{\alpha}}.
   \end{align*}
   If $|\omega|\geq 1$ the proof is the same of inequality \eqref{eq:bd_ups1} of Lemma \ref{L_uplow}, since for $|\omega|\geq 1$, $|\omega|^s\asymp (1+|\omega|)^s$.  
 \end{proof}

 We obtain the following Corollary.
 \begin{cor} 
 \label{cor:sommasemi}
 Under the same hypothesis \eqref{eq_fcndsemi}, there exists $ b_s \in \R$ such that
 \begin{align}
   \label{eq:upssemi}
    \sum_{j \in \N, k }\frac{2^{2js}}{|\omega|^{2s}} |\Phi_{j,k}(\omega)|^2\leq b_s, \quad \mbox{a.e. } \omega \in \R. 
 \end{align}	
\end{cor}
\begin{proof}
  The proof is a consequence of the inequalities proven in Lemma \ref{L_uplowsemi}, the scheme is the same of Corollary
  \ref{cor:sommanorm}.
\end{proof}

We state now the counterpart of Lemma \ref{lem:1Bessel} and Lemma \ref{lem:2Bessel} in the framework of seminorm discretization.

\begin{lem}
  \label{lem:1Besselsemi}
   Let $s\geq 0$ and $\varphi$  such that $\Phip(\omega)$ satisfy hypothesis \eqref{eq_fcndsemi}.
   We define
   \begin{align}
      \nonumber
      \widetilde{\phi_{j,k,\lambda}}(t)&=\frac{1}{2^{j/2}} T_{\frac{\lambda}{\bap}} 
      \F^{-1}_{\omega \mapsto t}\pt{ \frac{2^{js}}{|\omega|^s} \Phip(\omega)}(t), \quad j\in \N\\
      \label{eq:widetildesemi}
      \widetilde{\phi_{-j,k,\lambda}}(t)&=
      \frac{1}{2^{j/2}} T_{{\lambda}{\bap}} \Dil_{2^{2j}}
      \F^{-1}_{\omega \mapsto t}\pt{ \frac{2^{js}}{|\omega|^s} \Phip(\omega)}(t), \quad j\in \N\setminus \ptg{0}
   \end{align}
   then, for all $\nu\in (0,1]$ and  $j,k$ the system of functions 
   \[
     \ptg{ \widetilde{\phi_{j,k,\lambda}}(t)}_{\lambda\in \nu \Z}
   \]
   is a Bessel sequence uniformly in $j,k$, that is
   \begin{align}
     \label{eq:bessel_s}
     \sum_{\lambda \in \nu\Z} \abs{\Lprod{\widetilde{\phi_{j,k,\lambda}}} {f}}^2\leq C_\nu \norm{f}^2_{L^2(\R)}
   \end{align}
   with $C_\nu$ independent on $j,k$.
 \end{lem}
%%%%%%%%%%%%%%%%%%%%%%%%%%%%%%%%%%%%%%%%%%%%%%%%%%%%%%%%%%%%%%%%%%%%%%%%%%%%%%%%%%%%%%%%%%%%%%%%%%%%%%%%%%%%%%%%%%%%%%%%%%%%%%%%  
%%%%%%%%%%%%%%%%%%%%%%%%%%DIMOSTRAZIONE LEMMA 1 LIMITATEZZA OPERATORE DI COEFFICIENTI%%%%%%%%%%%%%%%%%%%%%%%%%%%%%%%%%%%%%%%%%%%          
%%%%%%%%%%%%%%%%%%%%%%%%%%%%%%%%%%%%%%%%%%%%%%%%%%%%%%%%%%%%%%%%%%%%%%%%%%%%%%%%%%%%%%%%%%%%%%%%%%%%%%%%%%%%%%%%%%%%%%%%%%%%%%%%  

 \begin{lem}
    \label{lem:2Besselsemi}
   Let $\Phip$, $\widetilde{\phi_{j,k,\lambda}}$ as in Lemma \ref{lem:1Besselsemi} 
   and $E_{j,k}$ as in \eqref{eq:epnu} for $j\in \N$ and as
   \[
	E_{-j, k}=\ptg{x\in \R \mid 2^{2j} x\in E_{j,k} }
   \]
   for negative integers.
   If $\nu\in (0,1]$ and 
   \begin{align*}
   \supp \widehat{f} \cap E_{j,k} =\emptyset,
   \end{align*} 
   then
   \begin{align}
     \label{eq:bessel2_s}
     \sum_{\lambda \in \nu \Z} \abs{\Lprod{\widetilde{\phi_{j,k,\lambda}}} {f}}^2\leq \frac{C_\nu}{2^{|j|(2\alpha -1)}} \norm{f}^2_{L^2(\R)}
   \end{align}
   with $C_\nu$ independent on $j,k$.
 \end{lem} 

\begin{proof}[Lemma \ref{lem:1Besselsemi}, \ref{lem:2Besselsemi}]
  If $j\in \N$ the proof of both Lemmata is the same of of Lemma \ref{lem:1Bessel} and \ref{lem:2Bessel}.
  In order to prove Lemma \eqref{lem:1Besselsemi} for $-j$ with $j\in \N\setminus \ptg{0}$, we use the following relation
  \begin{align*}
    & \sum_{\lambda \in \nu \Z} \abs{ \Lprod{\widetilde{\phi_{-j,k,\lambda}}} {f}}^2=
    \sum_{\lambda \in \nu \Z} \abs{\Lprod{ \Dil_{2^{2j}} T_{\frac{\lambda}{\beta(j)}}
   \widetilde{\phi_{j,k,0 }}} {f}}^2
   %=\\
   % &
%     \sum_{\lambda \in \nu \Z} 
    %  \abs{ \Lprod{T_{ \frac{\lambda}{ \beta(j)}}  
     %\widetilde{\phi_{j,k,0 }} } { \Dil_{2^{-2j}} f}}^2 \leq
    =\sum_{\lambda \in \nu \Z} \abs{ \Lprod{ T_{ \frac{\lambda} { \beta(j)} } 
     \widetilde{\phi_{j,k,0 }}} {\Dil_{2^{-2j}}{f}} }^2
  \end{align*}
  and the fact that $\norm{\Dil_{2^{-2j} }f  }_{L^2(\R)}= \norm{f}_{L^2(\R)}$.
  For Lemma \eqref{lem:2Besselsemi} notice also that 
  \(
       \supp \widehat{f} \cap E_{-j,k}=\emptyset      
  \)
  implies
  \(
    \supp \widehat{\Dil_{2^{-2j}} f}\cap E_{j,k}=\emptyset.      
  \)

\end{proof}
We can now prove the boundedness of the frame operator in the seminorm setting. 

\begin{thm}
    \label{thm:boundcoeffsemi}
    Let $s\geq 0$ and $\varphi$ so that condition \eqref{eq_fcndsemi} holds for $\alpha>\frac{1}{2}$, 
%       Moreover,
%       \[
%       \widehat\varphi(\xi) \lesssim (1+|\xi|)^{-\alpha}. 
%       \]
%        \item[iii)] 
%          \begin{align}
%          \label{eq:con3}  
%            |\Phip(\omega)|\lesssim 1, \quad \xi\in I_p.
%          \end{align}
%          
   then, 
   \begin{equation}
    \label{eq_coeff_opssemi}
    \sum_{j,k,\lambda \in \Delta}2^{2js}\abs{\Lprod{f}{\phipkn}}^2 \leq C_\nu  \abs{f}_s^2.
   \end{equation}   
   \end{thm}
%%%%%%%%%%%%%%%%%%%%%%%%%%%%%%%%%%%%%%%%%%%%%%%%%%%%%%%%%%%%%%%%%%%%%%%%%%%%%%%%%%%%%%%%%%%%%%%%%%%%%%%%%%%%%%%%%%%%%%%%%%%%%%%%  
%%%%%%%%%%%%%%%%%%%%%%%%%%DIMOSTRAZIONE LIMITATEZZA OPERATORE DI COEFFICIENTI%%%%%%%%%%%%%%%%%%%%%%%%%%%%%%%%%%%%%%%%%%%%%%%%%%% 
%%%%%%%%%%%%%%%%%%%%%%%%%%%%%%%%%%%%%%%%%%%%%%%%%%%%%%%%%%%%%%%%%%%%%%%%%%%%%%%%%%%%%%%%%%%%%%%%%%%%%%%%%%%%%%%%%%%%%%%%%%%%%%%%  
\begin{proof}
% For indexes $j\in \N$ we use the same idea of Theorem \ref{thm:Bound}.
% For negative indexes indexes notice that
% \begin{align*}
%   &2^{2(-j)s}\abs{\Lprod{f}{\varphi_{-j,k,\lambda}}}^2=
%2^{2(-j)s}   \abs{\Lprod{f}{\Dil_{2^{2j}} T_{\frac{\lambda}{ \bap}} \varphi_{j,k,0}}}^2 =\\
%   &2^{2(-j)s}\abs{\Lprod{\Dil_{2^{-2j}}f}{ T_{\frac{\lambda}{ \bap}} \varphi_{j,k,0}}}^2.
% \end{align*}
% Then, using the result for $j\in \N$, we get
% \[
% \sum_{j>0,k,\lambda \in \Delta}2^{2(-j)s}\abs{\Lprod{f}{\varphi_{-j,k,0}}}^2 = \sum_{j>0,k,\lambda \in \Delta}2^{2(-j)s}\abs{\Lprod{\Dil_{2^{-2j}}f}{ T_{\frac{\lambda}{ \bap}} \varphi_{j,k,0}}}^2
% \]
%and
%\[
%  \abs{\Dil_{2^{-2j}}f}_{s}^2= 2^{2js}\abs{f}_{s}^2.
%\]
We notice that using Plancherel Theorem one gets
\begin{align*}
     &\sum_{j,k,\lambda\in\Delta}4^{js}\abs{\Lprod{f}{\phipkn}}^2 =
     \sum_{j,k,\lambda\in\Delta}4^{js}\abs{\Lprod{\F (f)}{\F(\phipkn)}}^2\\
     &=\sum_{j,k,\lambda\in\Delta}\abs{\Lprod{\F (f) |\omega|^s}{ \frac{2^{js}}{|\omega|^s}\F(\phipkn)}}^2\\
     &=\sum_{j\geq0,k,\lambda\in\Delta}\abs{\Lprod{\F (f) |\omega|^s}{\frac{1}{2^{j/2}} e^{-\pii \frac{\lambda}{\bap}(\cdot)}\Phip(\cdot) \frac{2^{js}}{|\omega|^s} } }^2 \\
     &\quad+ \sum_{j>0,k,\lambda\in\Delta}\abs{\Lprod{\F (f) |\omega|^s}{\frac{1}{2^{j/2}} e^{-\pii \lambda\bap(\cdot)}D_{2^{2j}}\pt{\Phip(\cdot) \frac{2^{js}}{|\omega|^s}}}}^2 \\
         &=\sum_{j\geq0,k,\lambda\in \Delta}\abs{\Lprod{\F (f) |\omega|^s}{  \F( \widetilde{\phi_{j,k,\lambda}) }}}^2 + \sum_{j>0,k,\lambda\in \Delta}\abs{\Lprod{\F (f) |\omega|^s}{  \F( \widetilde{\phi_{-j,k,\lambda}) }}}^2,\\      
\end{align*}
which is the counterpart of ~\eqref{eq:thm3_last} in Theorem~\ref{thm:Bound}. 
Hence we can conclude as shown there using Lemmata~\ref{lem:1Besselsemi} and \ref{lem:2Besselsemi} instead of \ref{lem:1Bessel},\ref{lem:2Bessel}.   
\end{proof} 

\begin{proof}[Proof of Theorem \ref{thm:mainsemi}]
	The upper bound follows from Theorem \ref{thm:boundcoeffsemi}, while the lower bound 
	is proven as in Theorem \ref{thm_down_op}.
\end{proof}

  \section{Explicit examples of $H^s$-Frames} 
  \label{sec:ex}
 
 \label{s:expls}
 In this section we provide some explicit examples of frames which discretize the $H^s$-norm and of frames
 which discretize the $H^s$-seminorm.

 \subsection{Example 1}
   \label{subsect:orto}
 Let us consider $\varpb(t)= \sinc(t)=\pt{\F^{-1}\chi_{\pt{-\frac{1}{2},\frac{1}{2}}} )}(t)$ and
 $\varphi(t)=\varpb(t) $ and the partition introduced in \eqref{eq:setspl}; see Figure~\ref{Fig:Sinc}.
 Since the characteristic function
 has compact support, the couple $\varpb, \varphi$ is trivially $s$-admissible for all $s\geq 0$
 both for the Sobolev norm and for the Sobolve seminorm.
 \begin{figure}
     \centering
     \begin{subfigure}[b]{0.4\textwidth}
 \includegraphics[width=.8\textwidth]{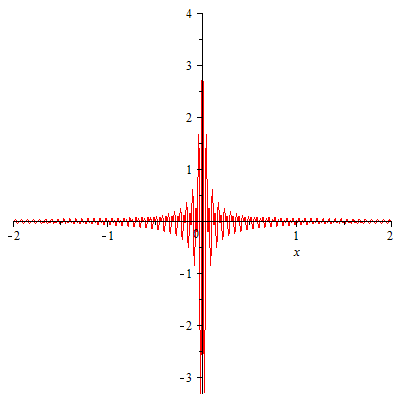}
         \caption{Time domain}
     \end{subfigure}
     \begin{subfigure}[b]{0.4\textwidth}
         \includegraphics[width=.8\textwidth]{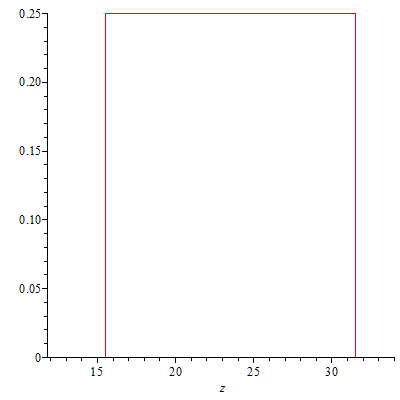}
         \caption{Frequency domain}
     \end{subfigure}
 \caption{Frame element $\phipkn(t)$ in both time and frequency. $\phi = \sinc(t)$ and $j=4$.}\label{Fig:Sinc}
 \end{figure}

 Therefore by Section \ref{sec:normHs}, for $\nu$ small enough, the system of function
 \begin{align}
 \nonumber
      \varphi_{j,+} (t)=&T_{\frac{\lambda}{2^j}}  \pt{   \frac{\lambda}{ \sqrt{\bap} } \sum_{\eta =2^{j}}^{2^{j+1}-1} 
   e^{\pii \eta t } \varphi(t)}, & \lambda \in \nu \Z,
   \\
   \nonumber
     \varphi_{j,-, \lambda} (t)=& T_{\frac{\lambda}{2^j}}  \pt{  \frac{1}{ \sqrt{\bap} } \sum_{\eta =2^{j}}^{2^{j+1}-1} 
   e^{-\pii \eta t } \varphi(t)}= \overline{\varphi_{j,+} (t)}, & \lambda \in \nu \Z, \quad j\in \N 
 \\  
  \label{eq:excar}
   \varphi_{\bullet,\lambda} (t) =& \varphi(t - \lambda) , \lambda \in \nu \Z\,
 \end{align}
 is a frame.

  This example is indeed very peculiar. 
 The system of function \eqref{eq:excar} is not only a frame for $\nu$ small enough but is an orthonormal base of
 $L^2(\R)$ if we set $\nu=1$. Indeed,  notice that
 \begin{align*}
   \Phi_{j,\pm}(\omega)&= \sum_{\eta=\pm 2^j}^{\pm(2^{j}-1)} \widehat{\varphi}(\omega-\eta)=
   \sum_{\eta= \pm 2^j}^{\pm(2^{j}-1)} \chi_{\pt{-\frac{1}{2}, \frac{1}{2}}}(\omega-\eta)= \chi_{I_{j,\pm}}(\omega),\\
 %   \Phi_{j,-}(\omega)&= \sum_{\eta=2^j}^{2^{j}-1} \widehat{\varphi}(\omega+\eta)=
 %   \sum_{\eta=2^j}^{2^{j}-1} \chi_{\pt{-\frac{1}{2}, \frac{1}{2}} }(\omega+\eta)= \chi_{I_{j,-}} (\omega),\\
   \Phi_{\bullet}(\omega)&=\chi_{(-\frac{1}{2},\frac{1}{2})}(\omega).
 \end{align*}
 Therefore, if $j\neq j'$  or $k\neq k'$
 \begin{align*}
   \Lprod{\varphi_{j,k,\lambda}}{ \varphi_{j',k',\lambda'}} &= 
   \Lprod{ \widehat{\varphi_{j,k,\lambda}}}{\widehat{ \varphi_{j',k',\lambda'}}}\\
   &=\Lprod{\frac{1}{2^{j/2}} e^{-\pii \cdot \lambda/2^{j}} \chi_{I_{j,k}}(\cdot) }{\frac{1}{2^{j'/2}} e^{-\pii \cdot \lambda/2^{j'}} \chi_{I_{j',k'}} (\cdot)}=0  
 \end{align*}
 for  all $\lambda, \lambda'\in \Z $.
 Moreover if $j=j'$ and $k=k'$, by well known properties of Fourier series,
 \begin{align*}
   \Lprod{\varphi_{j,k,\lambda}}{ \varphi_{j,k,\lambda'}} &= \\
    &=\Lprod{\frac{1}{2^{j/2}} e^{-\pii \cdot \lambda/2^{j}} \chi_{I_{j,k}}(\cdot) }{\frac{1}{2^{j/2}} e^{-\pii \cdot \lambda/2^{j}} \chi_{I_{j,k}} (\cdot)}=\delta_0(\lambda-\lambda') . 
 \end{align*}
 In the same way one shows that $\varphi_{\bullet,\lambda}$ is orthogonal to
 $\varphi_{j,k,\lambda}$ for all $j,k,\lambda$ and that
 \[
   \Lprod{\varphi_{\bullet,\lambda}}{\varphi_{\bullet,\lambda'}}=\delta_0(\lambda-\lambda').
 \]

 \begin{rem}
 The orthonormal basis system we introduced is strictly related to Shannon basis.
 As expected the localization properties in the time domain of this frame are not so 
 strong, due to the mild localization properties of the $\sinc$ function.
 Nevertheless, in our setting, we can gain localization considering 
 powers of $\sinc$. 
 That is we can consider $\phi(t)=\sinc(t)^n$ and this new window has increasing 
 localization as $n$-increases, moreover it is always $s$-admissible since its Fourier transform has compact support.
 In Figure \ref{Fig:FrWin} there are two examples with $n=4$ and $n=15$. 
 It is clear that as $n$ increases the function gains localization in time and looses a little of localization in the 
 frequency domain. This loss is nevertheless not so high, since we are not dilating the function but summing.
 \end{rem}

 \subsection{Example 2}
 The definition of $s$-admissible window function involves properties of
 $\Phi_{j,k}$ rather then on the window function $\varphi$.
 In general, it is very difficult to check the properties of $\Phi_{j,k}$ since the definition
 involves sums.
 Nevertheless, it is possible to provide sufficient conditions on $\varphi$ which guarantee the 
 $s$-admissibility.
 In the following we will always consider $\varphi=\varpb$.
 \begin{Prop}
   \label{prop:decay}
   Let $\varphi$ such that
   \begin{align}
     \label{eq:condphiex}
     \abs{\widehat{\varphi}(\omega)}\lesssim \frac{1}{\pt{1+|\omega|}^{s+1+\epsilon}}, \quad \epsilon>0.
   \end{align}
   then  $\Phi_{j,k}$ satisfy condition \eqref{eq:con2}.
  
 \end{Prop}
 \begin{proof}
   First let us check that $\Phi_{j,k}$
   is uniformly bounded. It is clear that all $\Phi_{j,k}$
   belong to $L^{\infty}(\R)$ by condition \eqref{eq:condphiex},
   therefore the issue it to determine a uniform bound.
   Condition \eqref{eq:condphiex} implies that $\widehat{\varphi}$ belongs to $L^1(\R)$.
   Therefore, for all $j$  and $k$ we have
   \begin{align*}
     \abs{\Phi_{j,k}}&\leq 
     \sum_{\eta\in Z_{j,k} } \abs{\varphi(\omega-\eta)}
     \\&\lesssim \sum_{\eta\in Z_{j,k}} \frac{1}{\pt{1+|\omega-\eta|}^{s+1+\epsilon}}\\
     \\ &\lesssim \int \frac{1}{\pt{1+|\omega|}^{s+1+\epsilon}} d\omega\lesssim 1.
   \end{align*}
   In order to prove the decay property, notice again that the issue is to have a uniform bound with respect to $j$.
   Moreover, since   $\Phi_{j,k}$ is uniformly bounded we can prove the inequality just when 
   $d(\omega, I_{j,k})> 2^{j/2}$. Under this hypothesis, by \eqref{eq:condphi} we can write
   \begin{align*}
     \abs{\Phi_{j,k}}\leq& \sum_{\eta\in Z_{j,k} }\abs{\varphi(\omega-\eta)}
                          \lesssim \sum_{\eta\in Z_{j,k} } \frac{1}{\pt{1+|\omega-\eta|}^{1+s+\epsilon}}\\
                          & \lesssim  2^{j} \frac{1}{   \pt{1+ d(\omega, I_{j,k})}^{1+s+\epsilon} }
                           \lesssim  2^{j/2} \frac{1}{\pt{1+d(\omega, I_{j,k})}^{s+\frac{1}{2}+\epsilon} }.
   \end{align*}
 \end{proof}
 %%%%%%%%%%%%%%%%%%%%%%%%%%%%%%%%%%%%%%%%%%%%%%%%%%%%%%%%%%%%%%%%%%%%%%%%%%%%%%%
 %%%%%%%%%%%%%LIMITATEZZA DA SOTTO%%%%%%%%%%%%%%%%%%%%%%%%%%%%%%%%%%%%%%%%%%%%%%
 %%%%%%%%%%%%%%%%%%%%%%%%%%%%%%%%%%%%%%%%%%%%%%%%%%%%%%%%%%%%%%%%%%%%%%%%%%%%%%%
 \begin{Prop}
   \label{prop:posit}
   Let $\varphi$ such that $\widehat{\phi}$ has definite sign and 
   \begin{align}
     \label{eq:posphi}
     |\widehat{\varphi}(\omega)|\geq a, \quad \omega \in I_{\bullet}
   \end{align}
   then
   $\Phi_{\bullet}$ and $\Phi_{j,k}$ satisfy condition \eqref{eq:cond2bullet} and \eqref{eq:cond2j} respectively.
 \end{Prop}
 \begin{proof}
   There is nothing to prove for $\Phi_{\bullet}$.
   For $\Phi_{j,k}$ notice that, by hypothesis
   \begin{align*}
     |\Phi_{j,k}(\omega)|=\sum_{\eta\in Z_{j,k} }\abs{\varphi(\omega-\eta)}
   \end{align*}
  and that for all $\omega\in I_{j,k}$ there exists $\bar{\eta} \in Z_{j,k}$ 
  such that $\omega-\bar{\eta}\in I_{\bullet}$, therefore
  \[
    |\Phi_{j,k}(\omega)|=\sum_{\eta\in Z_{j,k} }\abs{\varphi(\omega-\eta)}\geq 
    \abs{\varphi(\omega-\bar{\eta})}\geq a.
  \]

 \end{proof}
 It is now immediate to state the following Theorem.
 \begin{figure}
     \centering
     \begin{subfigure}[b]{0.42\textwidth}
 \includegraphics[width=\textwidth]{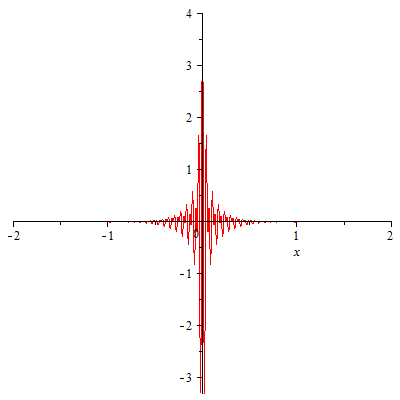}
         \caption{Time domain, $j=4$}
     \end{subfigure}
     \begin{subfigure}[b]{0.42\textwidth}
         \includegraphics[width=\textwidth]{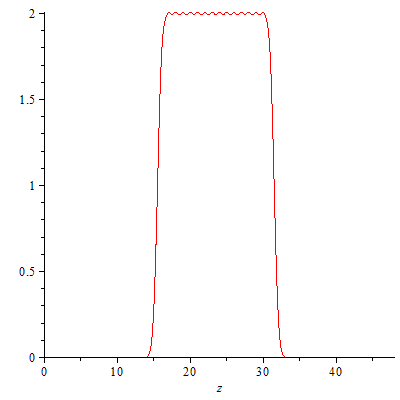}
         \caption{Frequency domain, $j=4$}
     \end{subfigure} 
 \caption{Time and frequency outlook of the real part of two frame elements with normalized Gaussian window.}
 \label{Fig:Gaus1}		
 \end{figure}
 \begin{thm}
   \label{thm:condsuff}
   Let $\varphi$ be a function such that
   \[
     \label{eq:condphi}
     \abs{\widehat{\varphi}(\omega)}\lesssim \frac{1}{\pt{1+\omega}^{s+1+\epsilon}}, \quad \epsilon>0.
   \]
   for a certain $s\in \R$ and such that $\widehat{\varphi}$ has definite sign and
   \[
         |\widehat{\varphi}(\omega)|\geq a, \quad \omega\in I_{\bullet}
   \]
   then the system of functions 
   \[
     \ptg{T_{\lambda} \varphi, \lambda \in \nu \Z} \cup\ptg{\varphi_{j,k,\lambda}, j,k,\lambda \in \Gamma}
   \]
   is a frame describing the $H^s$-norm for $\nu$ small enough.
 \end{thm}
 The easiest example of a function satisfying the conditions of Theorem \ref{thm:condsuff}
 is the Gaussian, see Figure ~\ref{Fig:Gaus1}. 
 % % % % % % % % % % % % % % % % % % % % % % % % % % % % % % % % % % % % % % 
 % % % % % % sinc in frequenza % % % % % % % % % % % % % % % % % % % % % % % 
 % % % % % % % % % % % % % % % % % % % % % % % % % % % % % % % % % % % % % % 
 \subsection{Example 3}
 In Theorem \ref{thm:condsuff} we provided necessary condition to obtain 
 admissible windows. In this subsection we describe in detail an example which shows that 
 the condition of Theorem \ref{thm:condsuff} are not necessary. 
 We will provide also another example of admissible function which describes the
 $H^s$-seminorm.
 %{\color{red} Forse omettiamo
 %We will provide somehow the dual of the orthonormal set introduces in Subsection \ref{subsect:orto}.
 %}
 Let us consider 
 \begin{align}
   \label{eq:sinc}
   \varphi=\chi_{\pt{-\frac{1}{2},\frac{1}{2}}}(t)=\pt{\F^{-1} 
   \sinc(\cdot)}(t)=\pt{\F^{-1} \frac{\sin(\pi \cdot )}{\pi \cdot}}(t).
 \end{align}
 It is clear that $\varphi$ defined in \eqref{eq:sinc} does not satisfy the sufficient conditions 
 \eqref{eq:condphiex} and \eqref{eq:posphi}. Nevertheless, it will provide a frame. See Figure~\ref{Fig:Caratteristica} for the plots of the frame with this particular window function.
 \begin{figure}
     \centering
     \begin{subfigure}[b]{0.4\textwidth}
 \includegraphics[width=\textwidth]{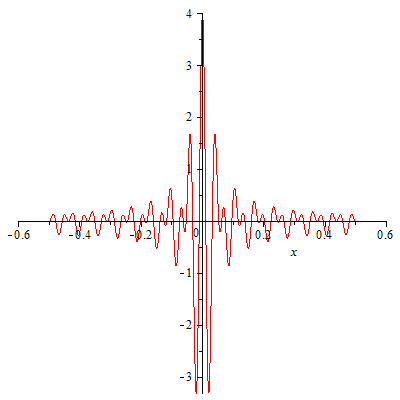}
         \caption{Time domain}
     \end{subfigure}
     \begin{subfigure}[b]{0.4\textwidth}
         \includegraphics[width=\textwidth]{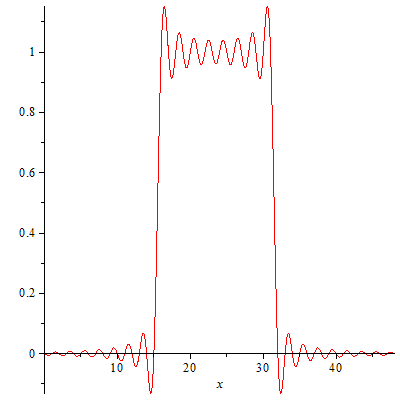}
         \caption{Frequency domain}
     \end{subfigure}
 \caption{Frame element $\phipkn(t)$ in both time and frequency. Here, $\phi = \chi_{\ptq{-\frac{1}{2},\frac{1}{2}}}$ and $j=4$.}\label{Fig:Caratteristica}
 \end{figure}

 \begin{Prop}
   \label{prop:sinc}
   If $\varphi$ is as in \eqref{eq:sinc}, then $\varphi$ is $s$-admissible for the Sobolev seminorm
   for each $s\in \left[0, 1\right)$.
 \end{Prop}
 \begin{proof}
 First we notice that by Remark~\ref{rem_decay_finite} it is enough to prove the decay property at infinity
 for $j>1$. %Since the polynomial behavior in $\omega=0$ is trivial, we do all our calculations for $j>1$.
 \medskip\\
   Let us suppose $k=+$, the case $k=-$ is equivalent.
   By definition
   \begin{align*}
     \Phi_{j,+}(\omega)&=\sum_{\eta=2^{j}}^{2^{j+1}-1}  \sinc(\omega-\eta)=\sum_{\eta=2^{j}}^{2^{j+1}-1} \frac{\sin \pt{\pi\pt{\omega-\eta}}}{\pi \pt{\omega-\eta}}.                   
   \end{align*}
   If $\omega\in I_{j,+}$ then, by construction, there exists $\bar\eta\in Z_{j,+}$ such that 
   $|\omega-\bar \eta|\leq \frac{1}{2}$, hence
   \[
     \abs{ \frac{ \sin\pt{\pi (\omega-\bar{\eta}) }}{\pi \pt{\omega-\bar{\eta}}} } \leq 1.
   \]
   We can write, using trigonometric inequalities,
       \begin{align*}
       \abs{\Phi_{j,k}(\omega)}&=\abs{\sum_{\eta =2^{j}}^{2^{j+1}-1} \frac{\sin \pt{\pi (\omega-\eta)}}{\pi(\omega-\eta)} }\\   
 %                               &\leq \abs{\sum_{\eta =2^{j}}^{\bar{\eta}-1} \frac{\sin \pt{\pi (\omega-\eta)}}{\pi(\omega-\eta)} }  
 %                                +\abs{\sum_{\eta =\bar{\eta}+1}^{2^{j+1}-1} \frac{\sin \pt{\pi (\omega-\eta)}}{\pi(\omega-\eta)} }  
 %                                + \abs{ \frac{ \sin\pt{\pi (\omega-\bar{\eta}) }}{\pi \pt{\omega-\bar{\eta}}} }\\
 %                               & \leq \abs{ \sin\pt{\pi(\omega-\bar{\eta})}
 %                               \sum_{m =1}^{\bar{\eta}-2^{j}} \frac{(-1)^m}{\pi(\omega-\eta+m)} }  
 %                               +\abs{ \sin\pt{\pi(\omega-\bar{\eta})}
 %                               \sum_{m =1}^{2^{j+1}-\bar{\eta}} \frac{(-1)^m}{\pi(\omega-\eta-m)} }
 %                               +1\\
                               &
                                \leq \abs{
                               \sum_{m =1}^{\bar{\eta}-2^{j}} \frac{(-1)^m}{\pi(\omega-\eta+m)} }  
                               +\abs{
                               \sum_{m =1}^{2^{j+1}-\bar{\eta}} \frac{(-1)^m}{\pi(\omega-\eta-m)} }
                               +1.
     \end{align*}
     Since the alternate harmonic series is convergent we can obtain an uniform bound with respect to $j$
     for $\Phi_{j,+}(\omega)$, if $\omega\in I_{j,+}$. Notice that, if $\omega\not \in I_{j,+}$ 
     the above inequality still holds, actually one could improve the bound, but this is not important for our purpose.
     In order to prove \eqref{eq:con2semi}, in view of the uniform bound of $\Phi_{j,+}$ we can suppose
     $d(w,I_{j,+})>2^{j/2}$.
     Therefore
       \begin{align*}
       \abs{\Phi_{j,k}(\omega)}&=%\abs{\sum_{\eta =2^{j}}^{2^{j+1}-1} \frac{\sin \pt{\pi (\omega-\eta)}}{\pi(\omega-\eta)} }\\   
                               \abs{\sum_{m =0}^{2^{j}-1} 
                               \frac{\sin \pt{\pi \omega} (-1)^m}{\pi(\omega-m-2^{j})} }  \\
                               %& \leq \abs{ \sin \pt{\pi \omega}\sum_{m =0}^{2^{j-1}-1} 
                              %\frac{1 }{\pi(\omega-m-2^{j})}-\frac{1 }{\pi(\omega-m-1-2^{j})} }\\
                               & \leq \abs{ \sin \pt{\pi \omega}\sum_{m =0}^{2^{j-1}-1} 
                               \frac{ 1 }{\pi(\omega-m-2^{j})(\omega-m-1-2^{j})} }\\
                               %&\leq 2^{j} 
                               %\abs{\sin \pt{\pi \omega} \frac{ 1 }{\pi d(\omega, I_{j,+})^2} }\\
                               &\leq 2^{j/2} 
                               \abs{\sin \pt{\pi \omega} \frac{ 1 }{\pi d(\omega, I_{j,+})^{\frac{3}{2}}} },
     \end{align*}
     which implies \eqref{eq:con2semi} for each $s\in [0,1)$.

     With the same notation as above, notice that
     \[
       \abs{ \frac{ \sin\pt{\pi (\omega-\bar{\eta}) }}{\pi \pt{\omega-\bar{\eta}}} } \geq \frac{2}{\pi}.
     \]
     Hence
     \begin{align*}
     \abs{\Phi_{j,k}(\omega)}&\geq \frac{2}{\pi} -\abs{ \sin\pt{\pi(\omega-\bar{\eta})}\pt{\sum_{m =1}^{\bar{\eta}-2^{j}} \frac{(-1)^m}{\pi(\omega-\eta+m)} 
                               +
                               \sum_{m =1}^{2^{j+1}-\bar{\eta}} \frac{(-1)^m}{\pi(\omega-\eta-m)}} }\\
                               &\geq \frac{2}{\pi}-\frac{1}{2}>0.
     \end{align*}
 \end{proof}
 \begin{rem}
 This example is closely related to the standard Haar basis. 
 It is well known that the Haar basis is suited to represent Sobolev spaces $H^s(\R)$ 
 with $s<\frac{1}{2}$, essentially due to the lack of continuity in the time domain which is related to
 the slow decay at infinity of the function $\sinc$.
 With our approach we can get rid of this problem and reach all $s<1$, the main reason is that
 performing sums instead of dilation we are able to exploit the oscillation behavior of the $\sinc$ function and
 increase the decay rate of functions $\Phi_{j,k}$.
 \end{rem}
 %Using the same idea of Proposition \ref{prop:sinc} we state similar results for the 
 %convolution of the characteristic functions.
 %
 %Let us denote  
 %\[
 %  \varphi^{(n\ast)}= \underbrace{\varphi \ast\ldots \ast \varphi}_{n-\mbox { times}}  
 %\]
 %so that $\varphi^{0\ast}=\varphi$.
 %\begin{Prop}
 %  \label{prop:sinc}
 %  If $\varphi$ is as in \eqref{eq:sinc}, then $\varphi^{2n\ast}$ is $s$-admissible for the Sobolev seminorm
 %  for each $s\in \left[0, 2n+1 \right)$.
 %\end{Prop}
 %\begin{figure}
 %    \centering
 %    \begin{subfigure}[b]{0.4\textwidth}
 %\includegraphics[width=\textwidth]{../immagini/1D2Funzcar.png}
 %        \caption{Time domain}
 %    \end{subfigure}
 %    \begin{subfigure}[b]{0.4\textwidth}
 %        \includegraphics[width=\textwidth]{../immagini/1D2FunzcarFourier.png}
 %        \caption{Frequency domain}
 %    \end{subfigure}
 %     \begin{subfigure}[b]{0.4\textwidth}
 %\includegraphics[width=\textwidth]{../immagini/1DFunzcarConvGauss.png}
 %        \caption{Time domain}
 %    \end{subfigure}
 %    \begin{subfigure}[b]{0.4\textwidth}
 %        \includegraphics[width=\textwidth]{../immagini/1DFunzcarConvGaussFourier.png}
 %        \caption{Frequency domain}
 %    \end{subfigure}
 %\caption{Convolution of $\chi_{\ptq{-\frac{1}{2},\frac{1}{2}}}$.We used $\varphi = \chi^{*3}_{\ptq{-\frac{1}{2},\frac{1}{2}}}$, for the first two figures and a convolution of $\chi_{\ptq{-\frac{1}{2},\frac{1}{2}}}$ and a Gaussian function in the last ones. The reference frequency is always $j=4$.}\label{Fig:Caratteristica}
 %\end{figure} 
 \section{Conclusion}
  In this paper we focused on the Sobolev properties of the Stockwell-like frame in arbitrary dimension
 obtaining a characterization of these spaces. 
 Although this is a very standard property that many other types of frames share, we believe that the extreme 
 flexibility of this frame opens interesting research paths.\medskip\\
 We are currently working on a generalization of the multi-dimensional partition presented in Section~\ref{sec:normHs} relaxing the condition of finitely many rotations
 that allows very different structures. For example, we can analyze a curvelet-like partition as showed 
 in Figure~\ref{fig:curvletes}. This particular tiling (both isotropic and anisotropic) offers very rich directional 
 information together with the usual parabolic scaling.\medskip\\
 It is clear from the result on the Sobolev seminorms, that this frame is very close to the wavelets one. 
 This motivates us to understand the structure behind our frame. In particular, we are analyzing a possible
 generalization of the Multi Resolution Analysis in this setting, discussing as well 
 sparsity properties.\\
 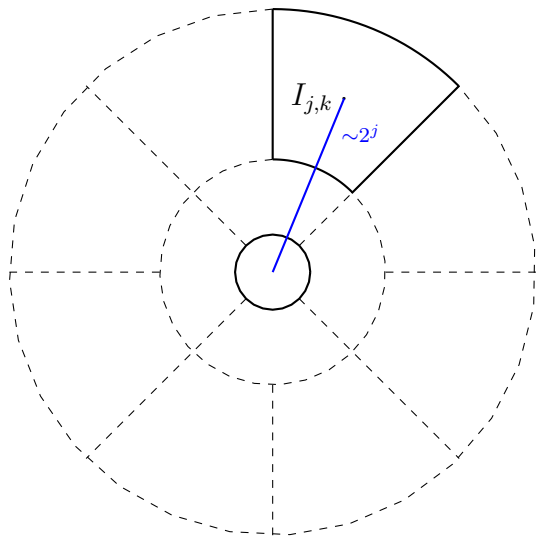
\begin{figure}
 \centering
 \begin{tikzpicture}[>=latex]
  \draw [dashed,domain=90:405] plot ({1.5*cos(\x)}, {1.5*sin(\x)});
   \draw [thick,domain=0:360] plot ({.5*cos(\x)}, {.5*sin(\x)});
   \draw [dashed,domain=90:405] plot ({3.5*cos(\x)}, {3.5*sin(\x)});
     \draw [thick,domain=45:90] plot ({1.5*cos(\x)}, {1.5*sin(\x)});
 \draw [thick,domain=45:90] plot ({3.5*cos(\x)}, {3.5*sin(\x)});
 \node at (.95,2.31) [left] {$I_{j,k}$};
 \node at (.95,2.31){.};
 \node at (.77,1.85) [right] {{\color{blue}$\scriptstyle \sim 2^j$}};
 \draw[dashed]  (.35,.35) -- (2.48,2.48);
 \draw[thick]  (1.05,1.05) -- (2.48,2.48);
 \draw[dashed]  (-.35,-.35) -- (-2.48,-2.48);
 \draw[dashed]  (.35,-.35) -- (2.48,-2.48);
 \draw[dashed]  (-.35,.35) -- (-2.48,2.48);
 \draw[thick]   (0,1.5) --(0,3.5);
 \draw[dashed]   (1.5,0) --(3.5,0);
 \draw[dashed]   (-1.5,0) --(-3.5,0);
 \draw[dashed]   (0,-1.5) --(0,-3.5);
 \draw[thick,blue]   (0,0)  --(.95,2.31) ;
 \end{tikzpicture} 
 \caption{Curvelet like decomposition of the frequency domain.}
 \label{fig:curvletes}
  \end{figure}   

 In Section~\ref{s:expls} we proposed several examples of frame windows; in particular, we considered
 the powers of the $\sinc$ function. 
 We are testing this frame numerically and the first evidences are very promising. For instance, this 
 window function provides great localization in space while being just a perturbation of the boxcar function 
 of the frequency band. These properties, usually leads to very sparse approximations.
 This frame does not provide an orthonormal basis, but it would be interesting to understand
 if it is tight.\medskip\\
 On the other hand it would be interesting to investigate the density of the Stockwell frame with 
 Gaussian window and compare it with well known results on Gabor frames. \medskip\\
 Concerning numerics,
 we are implementing a new algorithm for Stockwell-like frames that follows the work of \cite{BR16}
 and generalizes it to frames in various dimensions.
 \hspace{1em}\\
 \paragraph{\bf{Acknowledgments}}
 We thank Fabio Nicola  for the useful discussions on the subject.
 The first author is partially supported by the Research Project FIR (Futuro in Ricerca) 2013 \emph{Geometrical and qualitative aspects of PDE's}.

  \bibliography{Bib_Sob_Dost.bib}
%\nocite{*}
\bibliographystyle{abbrv}

\end{document}